\documentclass[11pt,reqno,twoside]{amsart}
\usepackage{amssymb,amsmath,amsthm,soul,color,paralist}
\usepackage{t1enc}
\usepackage[cp1250]{inputenc}
\usepackage{a4,indentfirst,latexsym}
\usepackage{comment}
\usepackage{graphics}
\usepackage{mathrsfs}
\usepackage{cite,enumitem,graphicx}
\usepackage[colorlinks=true,urlcolor=blue,
citecolor=red,linkcolor=blue,linktocpage,pdfpagelabels,
bookmarksnumbered,bookmarksopen]{hyperref}
\usepackage[english]{babel}
\usepackage[left=2.50cm,right=2.50cm,top=2.72cm,bottom=2.72cm]{geometry}
\usepackage[colorinlistoftodos]{todonotes}
\usepackage[normalem]{ulem}
\makeatletter
\providecommand\@dotsep{5}
\def\listtodoname{List of Todos}
\def\listoftodos{\@starttoc{tdo}\listtodoname}
\makeatother
\numberwithin{equation}{section}
\newcommand{\R}{\mathbb{R}}
\newcommand{\N}{\mathcal{N}}
\newcommand{\M}{\mathcal{M}}
\newcommand{\C}{\mathcal{C}}
\newcommand{\B}{\mathcal{B}}
\newcommand{\I}{\mathcal{I}}
\newcommand{\J}{\mathcal{J}}

\newcommand{\p}{p^{*}}
\newcommand{\q}{q^{*}}
\newcommand{\X}{\mathbb{X}}
\newcommand{\Y}{\mathbb{Y}}
\DeclareMathOperator{\dive}{div}
\DeclareMathOperator{\supp}{supp}
\DeclareMathOperator{\e}{\varepsilon}
\DeclareMathOperator{\ri}{\rightarrow}
\newtheorem{proposition}{Proposition}[section]
\newtheorem{lemma}{Lemma}[section]
\newtheorem{theorem}{Theorem}[section]

\newtheorem{remark}{Remark}[section]
\title[Multiplicity and concentration results  for a $(p, q)$-Laplacian problem]{Multiplicity and concentration results \\
for a $(p, q)$-Laplacian problem in $\R^{N}$}
\author[V. Ambrosio]{Vincenzo Ambrosio}
\address{Vincenzo Ambrosio\hfill\break\indent 
Dipartimento di Ingegneria Industriale e Scienze Matematiche \hfill\break\indent
Universit\`a Politecnica delle Marche\hfill\break\indent
Via Brecce Bianche, 12\hfill\break\indent
60131 Ancona (Italy)}
\email{v.ambrosio@univpm.it}
\author[D. Repov\v{s}]{Du\v{s}an Repov\v{s}}
\address{Du\v{s}an Repov\v{s}\hfill\break\indent
Faculty of Education, and Faculty of Mathematics and Physics \hfill\break\indent
University of Ljubljana\hfill\break\indent
\& Institute of Mathematics, Physics and Mechanics \hfill\break\indent
SI-1000 Ljubljana, Slovenia}
\email{dusan.repovs@guest.arnes.si}
\date{}
\begin{document}
\maketitle
\begin{abstract}
In this paper we study the multiplicity and concentration of positive solutions for the following $(p, q)$-Laplacian problem:
\begin{equation*}
\left\{
\begin{array}{ll}
-\Delta_{p} u -\Delta_{q} u +V(\e x) \left(|u|^{p-2}u + |u|^{q-2}u\right) = f(u) &\mbox{ in } \R^{N}, \\
u\in W^{1, p}(\R^{N})\cap W^{1, q}(\R^{N}), \quad u>0 \mbox{ in } \R^{N},
\end{array}
\right.
\end{equation*}
where $\e>0$ is a small parameter, $1< p<q<N$, $\Delta_{r}u=\dive(|\nabla u|^{r-2}\nabla u)$, with $r\in \{p, q\}$, is the $r$-Laplacian operator, $V:\R^{N}\ri \R$ is a continuous function satisfying the global Rabinowitz condition, and $f:\R\ri \R$ is a continuous function with subcritical growth. Using suitable variational arguments and Lyusternik-Shnirel'man category theory, we investigate the relation between the number of positive solutions and the topology of the set where $V$ attains its minimum for small $\e$. 
\end{abstract}

\section{Introduction}
\noindent
In this paper we deal with the existence and multiplicity of solutions for the following $(p,q)$-Laplacian problem:
\begin{equation}\label{P}\tag{$P_{\e}$}
\left\{
\begin{array}{ll}
-\Delta_{p} u -\Delta_{q} u +V(\e x) \left(|u|^{p-2}u + |u|^{q-2}u\right) = f(u) &\mbox{ in } \R^{N}, \\
u\in W^{1, p}(\R^{N})\cap W^{1, q}(\R^{N}), \quad 
u>0 &\mbox{ in } \R^{N},
\end{array}
\right.
\end{equation}
where $\e>0$ is a small parameter, $1< p<q<N$, $\Delta_{r}u=\dive(|\nabla u|^{r-2}\nabla u)$, with $r\in\{p, q\}$, is the $r$-Laplacian operator, $V:\R^{N}\ri \R$ is a continuous potential and $f:\R\ri \R$ is a continuous function with subcritical growth.

We recall that this class of problems arises from a general reaction-diffusion system 
\begin{align*}
u_{t}= \dive(D(u) \nabla u) + f(x, u) \quad x\in \R^{N}, t>0, 
\end{align*}
where $D(u)= |\nabla u|^{p-2}+ |\nabla u|^{q-2}$. As pointed out in \cite{Cherfils}, this equation appears in several applications such as biophysics, plasma physics and chemical reaction design. In these applications, $u$ describes a concentration, $\dive(D(u) \nabla u)$ corresponds to the diffusion with a diffusion coefficient $D(u)$, and the reaction term $f(x, u)$ relates to source and loss processes. 
Classical $(p, q)$-Laplacian problems in bounded or unbounded domains have been studied by several authors; see for instance \cite{AFans, FJMAA, FMN, HL, HLNA, LG, LL, PRR1} and references therein. 

In order to precisely state our result, we introduce the assumptions on the potential $V$ and the nonlinearity $f$. 
Along the paper we assume that
$V:\R^{N}\rightarrow \R$ is a continuous function satisfying the following condition introduced by Rabinowitz \cite{Rab}:
\begin{equation}\tag{$V$}\label{V0}
0<\inf_{x\in \R^{N}} V(x)=V_{0}< \liminf_{|x|\rightarrow \infty} V(x)=V_{\infty}\in (0, \infty],
\end{equation}
and the nonlinearity $f: \R\rightarrow \R$ fulfills the following hypotheses:
\begin{compactenum}
\item [$(f_{1})$] $f\in C^{0}(\R, \R)$ and $f(t)=0$ for all $t<0$;
\item [$(f_{2})$] $\displaystyle{\lim_{|t|\rightarrow 0} \frac{|f(t)|}{|t|^{p-1}}=0}$;
\item [$(f_{3})$] there exists $r\in (q, \q)$, with $\q=\frac{Nq}{N-q}$, such that $\displaystyle{\lim_{|t|\rightarrow \infty} \frac{|f(t)|}{|t|^{r-1}}=0}$;
\item [$(f_{4})$] there exists $\vartheta \in (q, \q)$ such that
\begin{equation*}
0<\vartheta F(t)= \vartheta \int_{0}^{t} f(\tau)\, d\tau \leq tf(t) \quad \mbox{ for all } t>0;
\end{equation*}
\item [$(f_{5})$] the map $\displaystyle{t\mapsto \frac{f(t)}{t^{q-1}}}$ is increasing on $(0, \infty)$.
\end{compactenum}

Since we deal with the multiplicity of solutions of \eqref{P}, we recall that if $Y$ is a given closed subset of a topological space $X$, we denote by $cat_{X}(Y)$ the Lyusternik-Shnirel'man category of $Y$ in $X$, that is the least number of closed and contractible sets in $X$ which cover $Y$ (see \cite{W} for more details).

Let us denote by
\begin{equation*}
M=\{x\in \R^{N} : V(x)=V_{0}\} \quad \mbox{ and } \quad M_{\delta}= \{x\in \R^{N} : dist(x, M)\leq \delta\}, \mbox{ for } \delta>0.
\end{equation*}
Our main result can be stated as follows:
\begin{theorem}\label{thmAI}
Assume that conditions $(V)$ and $(f_{1})$-$(f_{5})$ hold. Then for any $\delta>0$ there exists $\e_{\delta}>0$ such that, for any $\e \in (0, \e_{\delta})$, problem \eqref{P} has at least $cat_{M_{\delta}}(M)$ positive solutions. Moreover, if $u_{\e}$ denotes one of these solutions and $x_{\e}\in \R^{N}$ is a global maximum point of $u_{\e}$, then
$$
\lim_{\e\ri 0} V(\e x_{\e})=V_{0},
$$ 
and there exist $C_{1}, C_{2}>0$ such that
$$
u_{\e}(x)\leq C_{1}e^{-C_{2}|x-x_{\e}|} \quad \mbox{ for all } x\in \R^{N}.
$$
\end{theorem}

The proof of Theorem \ref{thmAI} will be obtained by using suitable variational techniques and category theory. 
We note that Theorem \ref{thmAI} improves Theorem $1.1$ in \cite{AFans}, in which the authors assumed $f\in C^{1}$ and that there exist $C>0$ and $\nu \in (p, \q)$ such that
\begin{equation*}
f'(t) t^{2} - (q-1) f(t)t \geq Ct^{\nu} \quad \mbox{ for all } t\geq 0. 
\end{equation*}
Since we require that $f$ is only continuous, the classical Nehari manifold arguments used in \cite{AFans} do not work in our context, and in order  to overcome the non-differentiability of the Nehari manifold, we take advantage of some  variants of critical point theorems from \cite{SW}. 
Clearly, with respect to \cite{AFans}, a more accurate and delicate analysis will be needed to implement our variational machinery. To obtain multiple solutions, we use a technique introduced by Benci and Cerami in \cite{BC}, which consists of making precise comparisons between the category of some sublevel sets of the energy functional $\mathcal{I}_{\e}$ associated with \eqref{P} and the category of the set $M$. Since we aim to apply Lyusternik-Shnirel'man theory, we need to prove certain compactness property for the functional $\mathcal{I}_{\e}$. In particular, we will see that the levels of compactness are strongly related to the behavior of the potential $V$ at infinity. This kind of argument has been recently employed by the first author for nonlocal fractional problems; see for example \cite{Aampa, AI}.
Finally, we prove the exponential decay of solutions by following some ideas from \cite{HL}. 
We would like to point out that our arguments are rather flexible and we believe that the ideas contained here can be applied in other situations to study problems driven by $(p, q)$-Laplacian operators, $\phi$-Laplacian operator, or also fractional $(p, q)$-Laplacian problems, on the entire space.
\smallskip

\noindent
The paper is organized as follows: in Section \ref{Sect2} we collect some facts about the involved Sobolev spaces  and some useful lemmas. In Section \ref{Sect3} we provide some technical results which will be crucial to prove our main theorem. In Section \ref{Sect4} we deal with the autonomous problems associated to \eqref{P}. In Section \ref{Sect5} we obtain an existence result for \eqref{P} for sufficiently small $\e$. Section \ref{Sect6} is devoted to the multiplicity result for \eqref{P}, and Section \ref{Sect7} to the concentration phenomenon.

\section{Preliminaries}\label{Sect2}

\noindent
In this section we recall some facts about the Sobolev spaces and we prove some technical lemmas which we will use later.

Let $p\in [1, \infty]$ and $A\subset \R^{N}$. We denote by $|u|_{L^{p}(A)}$ the $L^{p}(A)$-norm of a function $u:\R^{N}\rightarrow \R$ belonging to $L^{p}(A)$. When $A=\R^{N}$, we simply write $|u|_{p}$ instead of $|u|_{L^{p}(\R^{N})}$.
For $p\in (1, \infty)$ and $N>p$, we define $\mathcal{D}^{1, p}(\R^{N})$ as the closure of $C^{\infty}_{c}(\R^{N})$ with respect to
$$
|\nabla u|_{p}^{p}= \int_{\R^{N}} |\nabla u|^{p} dx.
$$
Let us denote by $W^{1, p}(\R^{N})$ the set of functions $u\in L^{p}(\R^{N})$ such that $|\nabla u|_{p}<\infty$, endowed with the natural norm
\begin{equation*}
\|u\|_{1, p}^{p}= |\nabla u|^{p}_{p}+ |u|_{p}^{p}.
\end{equation*}

We begin by recalling the following embedding theorem for Sobolev spaces.
\begin{theorem}(see \cite{Adams})\label{Sembedding}
Let $N>p$. Then there exists a constant $S_{*}>0$
such that, for any $u\in \mathcal{D}^{1, p}(\R^{N})$,
\begin{equation*}
|u|^{p}_{\p} \leq S_{*}^{-1} |\nabla u|^{p}_{p}.
\end{equation*}
Moreover, $W^{1, p}(\R^{N})$ is continuously embedded in $L^{t}(\R^{N})$ for any $t\in [p, p^{*}_{s}]$ and compactly in $L^{t}_{loc}(\R^{N})$ for any $t\in [1, p^{*})$.
\end{theorem}

\noindent
We recall the following Lions compactness lemma.
\begin{lemma}\label{Lions}(see \cite{Lions})
Let $N>p$ and $r\in [p, \p)$. If $\{u_{n}\}$ is a bounded sequence in $W^{1, p}(\R^{N})$ and if
\begin{equation}\label{ter4}
\lim_{n\rightarrow \infty} \sup_{y\in \R^{N}} \int_{\B_{R}(y)} |u_{n}|^{r} dx=0,
\end{equation}
where $R>0$, then $u_{n}\rightarrow 0$ in $L^{t}(\R^{N})$ for all $t\in (p, \p)$.
\end{lemma}

\noindent

We also have the following useful lemma.
\begin{lemma}\label{lemVince}(see \cite{AlvesNA, MeW})
Let $\eta_{n}: \R^{N}\ri \R^{K}$, $K\geq 1$, with $\eta_{n}\in L^{t}(\R^{N})\times \dots\times L^{t}(\R^{N})$ ($t>1$),
$\eta_{n}(x)\ri 0$ a.e. in $\R^{K}$ and $A(y)=|y|^{t-2}y$, $y\in \R^{K}$. Then, if $|\eta_{n}|_{t}\leq C$ for all $n\in \mathbb{N}$, we have
$$
\int_{\R^{N}} |A(\eta_{n}+w)-A(\eta_{n})-A(w)|^{t'}dx=o_{n}(1)
$$ 
for each $w\in L^{t}(\R^{N})\times \dots\times L^{t}(\R^{N})$ fixed, and $t'=\frac{t}{t-1}$ is the conjugate exponent of $t$.
\end{lemma}

For $\e>0$, we define the space
\begin{align*}
\X_{\e}=\left \{u\in W^{1, p}(\R^{N})\cap W^{1, q}(\R^{N}) : \int_{\R^{N}} V(\e x) \left(|u|^{p}+|u|^{q}\right) \,dx<\infty \right \}
\end{align*}
endowed with the norm
\begin{align*}
\|u\|_{\e}= \|u\|_{V, p} + \|u\|_{V, q},
\end{align*}
where
\begin{align*}
\|u\|_{V, t}^{t}= |\nabla u|_{t}^{t} + \int_{\R^{N}} V(\e x) |u|^{t}\, dx \quad \mbox{ for all } t>1.
\end{align*}
Then the following embedding lemma hold.
\begin{lemma}\label{embedding}(see \cite{AFans})
The space $\X_{\e}$ is continuously embedded into $W^{1, p}(\R^{N})\cap W^{1, q}(\R^{N})$.
Therefore $\X_{\e}$ is continuously embedded in $L^{t}(\R^{N})$ for any $t\in [p, \q]$ and compactly embedded in $L^{t}(B_{R})$, for all $R>0$ and any $t\in [1, \q)$.
\end{lemma}

\noindent
\begin{lemma}\label{lem6}(see \cite{AFans})
If $V_{\infty}=\infty$, the embedding $\X_{\e}\subset L^{m}(\R^{N})$ is compact for any $p\leq m<\q$.
\end{lemma}

Finally we have the following splitting lemma which will be very useful in this work.

\begin{lemma}\label{lem7}
Let $\{u_{n}\}\subset \X_{\e}$ be a sequence such that $u_{n}\rightharpoonup u$ in $\X_{\e}$. Set $v_{n}=u_{n}-u$.  Then we have
\begin{compactenum}[$(i)$]
\item $\displaystyle{|\nabla v_{n}|_{p}^{p}+|\nabla v_{n}|_{q}^{q}= \left(|\nabla u_{n}|_{p}^{p}+|\nabla u_{n}|_{q}^{q}\right) - \left(|\nabla u|_{p}^{p}+|\nabla u|_{q}^{q}\right)+o_{n}(1)}$,
\item $\displaystyle\int_{\R^{N}} V(\e x) \left(|v_{n}|^{p}+ |v_{n}|^{q}\right) \, dx= \int_{\R^{N}} V(\e x) \left(|u_{n}|^{p}+ |u_{n}|^{q}\right) \, dx- \int_{\R^{N}} V(\e x) \left(|u|^{p}+ |u|^{q}\right) \, dx +o_{n}(1)$
\item $\displaystyle{\int_{\R^{N}} \left(F(v_{n})- F(u_{n})+ F(u)\right) \, dx =o_{n}(1)}$,
\item $\displaystyle{\sup_{\|w\|_{\e}\leq 1} \int_{\R^{N}} |\left(f(v_{n}) - f(u_{n})+ f(u)\right) w|\, dx = o_{n}(1)}$.
\end{compactenum}
\end{lemma}
\begin{proof}
It is clear that $(i)$ and $(ii)$ are consequences of the well-known Brezis-Lieb lemma \cite{BL}. The proofs of $(iii)$ and $(iv)$ are given in \cite{AFans} for $f\in C^{1}$. Since here we are assuming $f\in C^{0}$, we need to use different arguments.
We start by proving $(iii)$. 
Let us note that $u_{n}=v_{n}+u$ and
$$
F(u_{n})-F(v_{n})=\int_{0}^{1} \frac{d}{dt} F(v_{n}+tu)\,dt=\int_{0}^{1} u f(v_{n}+tu)\,dt.
$$
In view of $(f_{2})$ and $(f_{3})$, for any $\delta>0$ there exists $c_{\delta}>0$ such that
\begin{align}
&|f(t)|\leq p\delta |t|^{p-1}+ c_{\delta}|t|^{\q-1} \quad \mbox{ for all } t\in \mathbb{R}, \label{Ftiodio} \\
&|F(t)|\leq \delta |t|^{p}+ c'_{\delta}|t|^{\q} \quad \mbox{ for all } t\in \mathbb{R} \label{Ftiodio1}.
\end{align}
Using \eqref{Ftiodio} with $\delta=1$ and $(|a|+|b|)^{r}\leq C(r)(|a|^{r}+ |b|^{r})$ for any $a, b\in \R$ and $r\geq 1$, we can see that
\begin{align}\label{odio}
|F(u_{n})-F(v_{n})|\leq C |v_{n}|^{p-1}|u|+C |u|^{p}+C |v_{n}|^{\q-1}|u|+C |u|^{\q}.
\end{align}
Fix $\eta>0$. Applying the Young inequality $ab \leq \eta a^{r} + C(\eta) b^{r'}$ for all $a, b> 0$, with $r, r'\in (1, \infty)$ such that $\frac{1}{r}+\frac{1}{r'}=1$, to the first and the third term on the right hand side of \eqref{odio}, we deduce that
$$
|F(u_{n})-F(v_{n})|\leq \eta (|v_{n}|^{p}+|v_{n}|^{\q})+C_{\eta} (|u|^{p}+|u|^{\q})
$$
which together with \eqref{Ftiodio1} with $\delta=\eta$ implies that
$$
|F(u_{n})-F(v_{n})-F(u)|\leq \eta (|v_{n}|^{p}+|v_{n}|^{\q})+C'_{\eta} (|u|^{p}+|u|^{\q}).
$$
Let
$$
G_{\eta, n}(x)=\max\left\{|F(u_{n})-F(v_{n})-F(u)|-\eta(|v_{n}|^{p}+|v_{n}|^{\q}), 0\right\}.
$$
Then $G_{\eta, n}\rightarrow 0$ a.e. in $\mathbb{R}^{N}$ as $n\rightarrow \infty$ (recall that $v_{n}\ri 0$ a.e. in $\R^{N}$ as $n\rightarrow \infty$), and $0\leq G_{\eta, n}\leq C'_{\eta} (|u|^{p}+|u|^{\q})\in L^{1}(\mathbb{R}^{N})$. As a consequence of the dominated convergence theorem we get
$$
\int_{\mathbb{R}^{N}} G_{\eta, n}(x) \, dx\rightarrow 0 \quad \mbox{ as } n\rightarrow \infty.
$$
On the other hand, by the definition of $G_{\eta, n}$, it follows that
$$
|F(v_{n})-F(u_{n})+F(u)|\leq \eta(|v_{n}|^{p}+|v_{n}|^{\q})+G_{\eta, n}
$$
which together with the boundedness of $(u_{n})$ in $L^{p}(\mathbb{R}^{N})\cap L^{\q}(\mathbb{R}^{N})$ yields
$$
\limsup_{n\rightarrow \infty} \int_{\mathbb{R}^{N}} |F(v_{n})-F(u_{n})+F(u)| \, dx\leq C\eta.
$$
By the arbitrariness of $\eta>0$ we can deduce that $(iii)$ holds. Finally, we prove $(iv)$. 
For any fixed $\eta>0$, by $(f_{2})$ we can choose $r_{0}=r_{0}(\eta)\in(0, 1)$ such that 
\begin{equation}\label{ZZ1}
|f(t)|\leq \eta |t|^{p-1} \quad \mbox{ for } |t|\leq 2r_{0}.
\end{equation} 
On the other hand, by $(f_{3})$ we can pick $r_{1}= r_{1}(\eta)>2$ such that
\begin{equation}\label{ZZ2}
|f(t)|\leq \eta |t|^{\q -1} \quad \mbox{ for } |t|\geq r_{1}-1.
\end{equation} 
By the continuity of $f$, there exists $\delta= \delta(\eta)\in (0, r_{0})$ satisfying 
\begin{align}\label{ZZ3}
|f(t_{1})- f(t_{2})|\leq r_{0}^{p-1}\eta \quad \mbox{ for } |t_{1}-t_{2}|\leq \delta, \, |t_{1}|, |t_{2}|\leq r_{1}+1. 
\end{align}
Moreover, by $(f_{3})$ there exists a positive constant $c=c(\eta)$ such that 
\begin{equation}\label{ZZ4}
|f(t)|\leq c(\eta) |t|^{p-1} + \eta |t|^{\q-1} \quad \mbox{ for all } t\in \R.
\end{equation} 
In what follows, we shall estimate the following term:
$$
\int_{\R^{N}\setminus B_{R}(0)} |f(u_{n}-u)- f(u_{n})-f(u)||w|\, dx.
$$
Using \eqref{ZZ4} and $u\in L^{p}(\R^{N})\cap L^{\q}(\R^{N})$, we can find $R=R(\eta)>0$ such that 
\begin{align*}
\int_{\R^{N}\setminus B_{R}(0)} |f(u) w|\, dx &\leq c\left(\int_{\R^{N}\setminus B_{R}(0)} |u|^{\q}\, dx \right)^{\frac{\q-1}{\q}} |w|_{\q}+ c \left(\int_{\R^{N}\setminus B_{R}(0)} |u|^{p}\, dx \right)^{\frac{p-1}{p}} |w|_{p}\\
&\leq c \eta \|w\|_{1, q} + c\eta \|w\|_{1, p}\leq c\eta \|w\|_{\e}. 
\end{align*} 
Set $A_{n}=\{x\in \R^{N}\setminus B_{R}(0) : |u_{n}(x)|\leq r_{0}\}$. Invoking \eqref{ZZ1} and applying the H\"older inequality we get 
\begin{align}\label{ZZ5}
\int_{A_{n}\cap \{|u|\leq \delta\}} |f(u_{n})- f(u_{n}-u)||w|\, dx \leq \eta (|u_{n}|_{p}^{p-1} + |u_{n}-u|_{p}^{p-1})|w|_{p} \leq c \eta \|w\|_{\e}. 
\end{align}
Let $B_{n}= \{x\in \R^{N} \setminus B_{R}(0) : |u_{n}(x)|\geq r_{1}\}$. Then \eqref{ZZ2} and the H\"older inequality yield
\begin{align}\label{ZZ6}
\int_{B_{n}\cap \{|u|\leq \delta\}} |f(u_{n})- f(u_{n}-u)| |w| \, dx \leq \eta (|u_{n}|_{\q}^{\q-1} + |u_{n}-u|_{\q}^{\q-1}) |w|_{\q} \leq c\eta \|w\|_{\e}. 
\end{align}
Finally, define $C_{n}= \{x\in \R^{N} \setminus B_{R}(0) : r_{0}\leq |u_{n}(x)|\leq r_{1}\}$. Since $u_{n}\in W^{1, p}(\R^{N})$ it follows that $|C_{n}|<\infty$. Now \eqref{ZZ3} gives
\begin{align}\label{ZZ7}
\int_{C_{n}\cap \{|u|\leq \delta\}} |f(u_{n})- f(u_{n}-u)| |w| \, dx \leq r_{0}^{p-1}\eta |w|_{p} |C_{n}|^{\frac{p-1}{p}}\leq \eta |u_{n}|_{p} |w|_{p} \leq c \eta \|w\|_{\e}.    
\end{align}
Putting together \eqref{ZZ5}, \eqref{ZZ6} and \eqref{ZZ7}, we obtain that 
\begin{align}\label{ZZ8}
\int_{(\R^{N}\setminus B_{R}(0))\cap \{|u|\leq \delta\}} |f(u_{n})- f(u_{n}-u)| |w| \, dx \leq c \eta \|w\|_{\e} \quad \mbox{ for all } n\in \mathbb{N}.    
\end{align}
Next, we note that \eqref{ZZ4} implies 
\begin{align*}
|f(u_{n})- f(u_{n}-u)|\leq \eta (|u_{n}|^{\q-1} + |u_{n}-u|^{\q-1}) + c(\eta) (|u_{n}|^{p-1}+ |u_{n}-u|^{p-1}), 
\end{align*}
so we can see that  
\begin{align*}
&\int_{(\R^{N}\setminus B_{R}(0)) \cap \{|u|\geq \delta\}} |f(u_{n})- f(u_{n}-u)||w|\, dx \\
&\leq \int_{(\R^{N}\setminus B_{R}(0)) \cap \{|u|\geq \delta\}} \left[ \eta (|u_{n}|^{\q-1} + |u_{n}-u|^{\q-1}) |w|+ c(\eta) (|u_{n}|^{p-1}+ |u_{n}-u|^{p-1})|w| \right]\, dx\\
&\leq c\eta \|w\|_{\e} +  \int_{(\R^{N}\setminus B_{R}(0)) \cap \{|u|\geq \delta\}} c(\eta) (|u_{n}|^{p-1}+ |u_{n}-u|^{p-1})|w|\, dx. 
\end{align*}
Since $u\in W^{1, p}(\R^{N})$, we get $|(\R^{N}\setminus B_{R}(0)) \cap \{|u|\geq \delta\}|\ri 0$ as $R\ri \infty$. Then choosing $R=R(\eta)$ large enough we can infer
\begin{align*}
&\int_{(\R^{N}\setminus B_{R}(0)) \cap \{|u|\geq \delta\}} c(\eta) (|u_{n}|^{p-1}+ |u_{n}-u|^{p-1})|w|\, dx \\
&\quad \leq c(\eta) (|u_{n}|_{\q}^{p-1}+ |u_{n}-u|_{\q}^{p-1}) \,|w|_{\q} \, |(\R^{N}\setminus B_{R}(0)) \cap \{u\geq \delta\}|^{\frac{\q-p}{p}}\leq \eta \|w\|_{\e}, 
\end{align*}
where we have used the generalized H\"older inequality. Therefore
\begin{align*}
\int_{(\R^{N}\setminus B_{R}(0)) \cap \{|u|\geq \delta\}} |f(u_{n})- f(u_{n}-u)||w|\, dx \leq c \eta \|w\|_{\e} \quad \mbox{ for all } n\in \mathbb{N}, 
\end{align*}
which combined with \eqref{ZZ8} yields
\begin{align}\label{ZZ17}
\int_{\R^{N}\setminus B_{R}(0)} |f(u_{n})-f(u)- f(u_{n}-u)||w|\, dx \leq c \eta \|w\|_{\e} \quad \mbox{ for all } n\in \mathbb{N}. 
\end{align}
Now, recalling that $u_{n}\rightharpoonup u$ in $W^{1, p}(\R^{N})$, we may assume that, up to a subsequence, $u_{n}\ri u$ strongly converges in $L^{p}(B_{R}(0))$ and there exists $h\in L^{p}(B_{R}(0))$ such that $|u_{n}(x)|, |u(x)|\leq |h(x)|$ for a. e. $x\in B_{R}(0)$. 

It is clear that 
\begin{align}\label{ZZ18}
\int_{B_{R}(0)} |f(u_{n}-u)||w|\, dx \leq c\eta \|w\|_{\e}
\end{align}
provided that $n$ is big enough. Let us define $D_{n}=\{x\in B_{R}(0) : |u_{n}(x) - u(x)|\geq 1\}$. Thus
\begin{align*}
\int_{D_{n}} |f(u_{n})- f(u)||w|\, dx &\leq \int_{D_{n}} \left( c(\eta) (|u|^{p-1}+ |u_{n}|^{p-1}) + \eta (|u_{n}|^{\q-1} + |u|^{\q-1})\right) |w|\, dx \\
&\leq c\eta \|w\|_{\e} + 2c(\eta) \int_{D_{n}} |h|^{p-1}|w|\, dx \\
&\leq c\eta \|w\|_{\e} + 2c(\eta) \left( \int_{D_{n}} |h|^{p} \, dx\right)^{\frac{p-1}{p}} |w|_{p}. 
\end{align*}
Observing that $|D_{n}|\ri 0$ as $n\ri \infty$, we can deduce that
\begin{align}\label{ZZ19}
\int_{D_{n}} |f(u_{n})- f(u)||w| \, dx \leq c\eta \|w\|_{\e}. 
\end{align}
Since $u\in W^{1, p}(\R^{N})$, we know that $|\{|u|\geq L\}|\ri 0$ as $L\ri \infty$, so there exists $L= L(\eta)>0$ such that for all $n$ 
\begin{align}\label{ZZ20}
&\int_{(B_{R}(0)\setminus D_{n})\cap \{|u|\geq L\}} |f(u_{n})- f(u)||w|\, dx \nonumber \\
&\quad \leq \int_{(B_{R}(0)\setminus D_{n})\cap \{|u|\geq L\}} \left[\eta (|u_{n}|^{\q-1} + |u|^{\q-1})|w| + c(\eta) (|u_{n}|^{p-1}+ |u|^{p-1})|w| \right]\, dx \nonumber  \\
&\quad \leq c\eta \|w\|_{\e}+ c(\eta) (|u_{n}|_{\q}^{p-1} + |u|_{\q}^{p-1}) \,|w|_{\q}\, |(B_{R}(0)\setminus D_{n})\cap \{|u|\geq L\}|^{\frac{\q-p}{p}} \nonumber \\
&\quad \leq c \eta \|w\|_{\e}. 
\end{align}
On the other hand, by the dominated convergence theorem we can infer 
\begin{align*}
\int_{(B_{R}(0)\setminus D_{n})\cap \{|u|\leq L\}} |f(u_{n})- f(u)|^{p}\, dx \ri 0 \quad  \mbox{ as } n\ri \infty. 
\end{align*}
Consequently,
\begin{align}\label{ZZ21}
\int_{(B_{R}(0)\setminus D_{n})\cap \{|u|\leq L\}} |f(u_{n})- f(u)| |w|\, dx\leq c \eta \|w\|_{\e} 
\end{align}
for $n$ large enough. Putting together \eqref{ZZ19}, \eqref{ZZ20} and \eqref{ZZ21}, we have 
\begin{align*}
\int_{B_{R}(0)} |f(u_{n})- f(u)| |w|\, dx \leq c \eta \|w\|_{\e}. 
\end{align*}
This and \eqref{ZZ18} yield
\begin{align}\label{ZZ22}
\int_{B_{R}(0)} |f(u_{n})- f(u)- f(u_{n}-u)| |w| \, dx \leq c \eta \|w\|_{\e}. 
\end{align}
Taking into account \eqref{ZZ17} and \eqref{ZZ22}, we can conclude that for $n$ large enough
\begin{align*}
\int_{\R^{N}} |f(u_{n})- f(u)- f(u_{n}-u)| |w|\, dx \leq c \eta \|w\|_{\e}. 
\end{align*}
This completes the proof of lemma.
\end{proof}

\section{Functional setting}\label{Sect3}

\noindent
In this section we consider the following problem
\begin{equation}\tag{$P_{\e}$}\label{Pe}
\left\{
\begin{array}{ll}
-\Delta_{p} u -\Delta_{q} u + V(\e x) \left(|u|^{p-2}u + |u|^{p-2}u\right) = f(u) &\mbox{ in } \R^{N},\\
u\in W^{1, p}(\R^{N}) \cap W^{1, q}(\R^{N}), \quad u>0 \mbox{ in } \R^{N}.
\end{array}
\right.
\end{equation}

\noindent
In order to study \eqref{Pe}, we look for critical points of the functional $\I_{\e}: \X_{\e}\rightarrow \R$ defined as 
\begin{equation*}
\I_{\e}(u)=\frac{1}{p} |\nabla u|_{p}^{p} + \frac{1}{q}|\nabla u|_{q}^{q}+ \int_{\R^{N}} V(\e x) \left( \frac{1}{p}|u|^{p}+ \frac{1}{q}|u|^{q}\right)\, dx - \int_{\R^{N}} F(u) \, dx. 
\end{equation*}

\noindent
It is easy to see that $\I_{\e}\in C^{1}(\X_{\e}, \R)$ and its differential is given by
\begin{align*}
\langle \I'_{\e}(u), \varphi \rangle &= \int_{\R^{N}} |\nabla u|^{p-2}\nabla u\cdot \nabla \varphi \,dx+ \int_{\R^{N}} |\nabla u|^{q-2}\nabla u\cdot \nabla \varphi \,dx\ \\
&+ \int_{\R^{N}} V(\e x) (|u|^{p-2} u\,  +|u|^{q-2} u)\, \varphi \,dx - \int_{\R^{N}} f(u)\varphi \, dx
\end{align*}
for any $u, \varphi \in \X_{\e}$. 
Now, let us introduce the Nehari manifold associated to $\I_{\e}$, that is
\begin{equation*}
\N_{\e}= \left\{u\in \X_{\e}\setminus \{0\} : \langle \I'_{\e}(u), u\rangle =0  \right\}, 
\end{equation*}
and define
$$
c_{\e}= \inf_{u\in \N_{\e}} \I_{\e}(u).
$$ 
Let us note that $\I_{\e}$ possesses a mountain pass geometry \cite{AR}.
\begin{lemma}\label{lem2.1}
The functional $\I_{\e}$ satisfies the following conditions:
\begin{compactenum}[$(i)$]
\item there exist $\alpha, \rho >0$ such that $\I_{\e}(u)\geq \alpha$ with $\|u\|_{\e}=\rho$;
\item there exists $e\in \X_{\e}$ with $\|e\|_{\e}>\rho$ such that $\I_{\e}(e)<0$.
\end{compactenum}
\end{lemma}

\begin{proof}
$(i)$ Using $(f_2)$ and $(f_3)$, for any given $\xi>0$ there exists $C_{\xi}>0$ such that
\begin{align}
&|f(t)|\leq \xi |t|^{p-1}+ C_{\xi} |t|^{r-1} \quad \mbox{ for any } t\in \R \label{growthf}, \\
&|F(t)|\leq \frac{\xi}{p}|t|^{p}+ \frac{C_{\xi}}{r}|t|^{r} \quad \mbox{ for any } t\in \R. \label{growthF}
\end{align}
Hence, taking $\xi \in (0, V_{0})$, we have 
\begin{align*}
\I_{\e}(u) &\geq \frac{1}{p}\|u\|_{V, p}^{p} + \frac{1}{q} \|u\|_{V, q}^{q} -\frac{\xi}{p} |u|_{p}^{p} - \frac{C_{\xi}}{r}|u|_{r}^{r}\\
&\geq C_{1} \|u\|_{V, p}^{p}  + \frac{1}{q} \|u\|_{V, q}^{q} - C'_{\xi}\|u\|_{\e}^{r}.
\end{align*}
Choosing $\|u\|_{\e}= \rho\in (0, 1)$ and using $1<p<q$, we have $\|u\|_{V, p}<1$ and therefore $\|u\|_{V, p}^{p}\geq \|u\|_{V, p}^{q}$ which combined with $a^{t}+b^{t}\geq C_{t}(a+b)^{t}$ for any $a, b\geq 0$ and $t>1$, yields
\begin{align*}
\I_{\e}(u) \geq C \|u\|_{\e}^{q}- C'_{\xi}\|u\|_{\e}^{r}. 
\end{align*}
Since $r>q$ we can find $\alpha>0$ such that $\I_{\e}(u)\geq \alpha >0$ for $\|u\|_{\e}= \rho$. 

\noindent
$(ii)$ By $(f_4)$ we can infer
\begin{equation*}
F(t)\geq C_{1}|t|^{\vartheta} - C_{2} \quad \mbox{ for any } t\geq 0,
\end{equation*}
for some $C_{1}, C_{2}>0$. Taking $v \in C^{\infty}_{c}(\R^{N})$ such that $v\geq 0$, $v \not \equiv 0$, we have
\begin{equation*}
\I_{\e}(tv)\leq \frac{t^{p}}{p} \|v\|_{\e}^{p} + \frac{t^{q}}{q} \|v\|_{\e}^{q}- t^{\vartheta}C_{1} \int_{\supp v} v^{\vartheta} dx + C_{2}|\supp v|\rightarrow -\infty \mbox{ as } t\rightarrow \infty.
\end{equation*}
\end{proof}

Now, in view of Lemma \ref{lem2.1}, we can use a version of mountain pass theorem without the Palais-Smale condition \cite{W} to deduce the existence of a $(PS)$-sequence $\{u_{n}\}$ at level $c'_{\e}$, namely 
\begin{align*}
\I_{\e}(u_{n})\ri c'_{\e} \quad \mbox{ and } \quad \I'_{\e}(u_{n})\ri 0, 
\end{align*}
where $c'_{\e}$ is the mountain pass level of $\I_{\e}$ defined as
\begin{align*}
c'_{\e}= \inf_{\gamma \in \Gamma} \max_{t\in [0, 1]} \I_{\e}(\gamma(t)),
\end{align*}
and $\Gamma= \{\gamma \in C^{0}([0, 1], \X_{\e}) \, : \, \gamma(0)=0,  \,\I_{\e}(\gamma(1))<0\}$. 

\begin{lemma}\label{newlemma}
The following holds 
$$
c'_{\e}=c_{\e}=\inf_{u\in \X_{\e} \setminus \{0\}} \max_{t\geq 0} \I_{\e}(tu).
$$
\end{lemma}
\begin{proof}
For each $u\in \mathbb{X}_{\e}\setminus\{0\}$ and $t>0$, let us introduce the function $h(t)=\I_{\e}(tu)$. Following the same arguments as in the proof of Lemma \ref{lem2.1} we deduce that 
$h(0)=0$, $h(t)<0$ for $t$ sufficiently large and $h(t)>0$ for $t$ sufficiently small. Hence, $\max_{t\geq 0} h(t)$ is achieved at $t=t_{u}>0$ satisfying $h'(t_{u})=0$ and $t_{u}u\in \N_{\e}$. 

Note that, if $u\in \N_{\e}$ then $u^{+}\neq 0$. Indeed, from $(f_{1})$, we can deduce that 
\begin{align*}
\|u\|^{p}_{V, p} + \|u\|^{q}_{V, q}= \int_{\R^{N}} f(u)u \, dx =\int_{\R^{N}} f(u^{+})u^{+} \, dx. 
\end{align*} 
Now, if $u^{+}\equiv 0$, then $\|u\|^{p}_{V, p} + \|u\|^{q}_{V, q}=0$, that is $u\equiv 0$, and this is a contradiction in view of $u\in\N_{\e}$. 

Next, we prove that $t_{u}$ is the unique critical point of $h$. Assume by contradiction that there exist $t_{1}$ and $t_{2}$ such that $t_{1}u, t_{2}u\in \N_{\e}$, that is
\begin{align*}
t_{1}^{p-q}|\nabla u|_{p}^{p} +|\nabla u|_{q}^{q}+ t_{1}^{p-q} \int_{\R^{N}} V(\e x) |u|^{p} \, dx + \int_{\R^{N}} V(\e x) |u|^{q}\, dx=\int_{\{u>0\}} \frac{f(t_{1}u)}{(t_{1}u)^{q-1} }u^{q}\, dx
\end{align*}
and
\begin{align*}
t_{2}^{p-q}|\nabla u|_{p}^{p} + |\nabla u|_{q}^{q}+ t_{2}^{p-q}\int_{\R^{N}} V(\e x) |u|^{p} \, dx + \int_{\R^{N}} V(\e x) |u|^{q}\, dx=\int_{\{u>0\}} \frac{f(t_{2}u)}{(t_{2}u)^{q-1} }u^{q}\, dx.
\end{align*}
Subtracting term by term in the above equalities we get
\begin{align*}
(t_{1}^{p-q}-t_{2}^{p-q})|\nabla u|_{p}^{p} + (t_{1}^{p-q}-t_{2}^{p-q}) \int_{\R^{N}} V(\e x)|u|^{p}\, dx = \int_{\{u>0\}} \left[\frac{f(t_{1}u)}{(t_{1}u)^{q-1} } - \frac{f(t_{2}u)}{(t_{2}u)^{q-1} } \right] u^{q} dx. 
\end{align*}
Now, if $t_{1}<t_{2}$, from $(f_{5})$ and recalling that $p<q$, we can infer 
\begin{align*}
0<(t_{1}^{p-q}-t_{2}^{p-q})|\nabla u|_{p}^{p} + (t_{1}^{p-q}-t_{2}^{p-q}) \int_{\R^{N}} V(\e x)|u|^{p}\, dx = \int_{\{u>0\}} \left[\frac{f(t_{1}u)}{(t_{1}u)^{q-1} } - \frac{f(t_{2}u)}{(t_{2}u)^{q-1} } \right] u^{q} dx<0, 
\end{align*}
which gives a contradiction.
Now we can argue as in \cite{W} to complete the proof. 
\end{proof}

Next, we prove the following useful result.
\begin{lemma}\label{lemB}
Let $\{u_{n}\}$ be a Palais-Smale sequence of $\I_{\e}$ at level $c$. Then 
\begin{compactenum}[$(i)$]
\item $\{u_{n}\}$ is bounded in $\X_{\e}$. 
\item $u_{n}^{-}\ri 0$ in $\X_{\e}$ and we may assume that $u_{n}\geq 0$ for any $n\in \mathbb{N}$. 
\end{compactenum}
\end{lemma}

\begin{proof}
$(i)$ From $(f_{4})$ we have
\begin{align*}
C(1+\|u_{n}\|_{\e}) &\geq \I_{\e}(u_{n})- \frac{1}{\vartheta} \langle \I'_{\e}(u_{n}), u_{n} \rangle \\
&= \left( \frac{1}{p}- \frac{1}{\vartheta}\right) \|u_{n}\|_{V, p}^{p} + \left( \frac{1}{q}- \frac{1}{\vartheta}\right) \|u_{n}\|_{V, q}^{q} + \frac{1}{\vartheta} \int_{\R^{N}} \left( f(u_{n})u_{n}- \vartheta F(u_{n})\right) \, dx \\
&\geq \left( \frac{1}{p}- \frac{1}{\vartheta}\right) \|u_{n}\|_{V, p}^{p} + \left( \frac{1}{q}- \frac{1}{\vartheta}\right) \|u_{n}\|_{V, q}^{q}\\
&\geq \left(\frac{1}{q}- \frac{1}{\vartheta}\right) (\|u_{n}\|_{V, p}^{p} + \|u_{n}\|_{V, q}^{q}). 
\end{align*}

Now, assume by contradiction that $\|u_{n}\|_{\e}\rightarrow \infty$. We shall distinguish among the following cases: \\
{{\rm Case $1$.}} $\|u_{n}\|_{V, p}\ri \infty$ and $\|u_{n}\|_{V, q}\ri \infty$.  

\noindent
Since $p<q$, we have, for $n$ sufficiently large, that $\|u_{n}\|_{V, q}^{q-p}\geq 1$, that is $\|u_{n}\|_{V, q}^{q}\geq \|u_{n}\|_{V, q}^{p}$, and thus
\begin{align*}
C(1+ \|u_{n}\|_{\e})&\geq \left(\frac{1}{q}- \frac{1}{\vartheta}\right) \left(\|u_{n}\|_{V, p}^{p} + \|u_{n}\|_{V, q}^{p} \right)\\
&\geq C_{1} \left(\|u_{n}\|_{V, p} + \|u_{n}\|_{V, q} \right)^{p} =C_{1} \|u_{n}\|_{\e}^{p},  
\end{align*}
which gives a contradiction.

\noindent
{{\rm Case $2$.}}
$\|u_{n}\|_{V, p}\ri \infty$ and $\|u_{n}\|_{V, q}$ is bounded. 

\noindent
We can see that 
\begin{align*}
C\left( 1+ \|u_{n}\|_{V, p} + \|u_{n}\|_{V, q} \right)\geq \left(\frac{1}{q}- \frac{1}{\vartheta}\right) \|u_{n}\|_{V, p}^{p}
\end{align*}
implies
\begin{align*}
C\left( \frac{1}{\|u_{n}\|_{V, p}^{p}}+ \frac{1}{\|u_{n}\|_{V, p}^{p-1}} + \frac{\|u_{n}\|_{V, q}}{\|u_{n}\|_{V, p}^{p}} \right)\geq \left(\frac{1}{q}- \frac{1}{\vartheta}\right), 
\end{align*}
and letting $n\ri \infty$, we get $0\geq \left(\frac{1}{q}- \frac{1}{\vartheta}\right) >0$, which yields a contradiction.  

\noindent
$\|u_{n}\|_{V, p}$ is bounded and $\|u_{n}\|_{V, q}\ri \infty$. 

\noindent
{{\rm Case $3$.}} We can proceed similarly as in the case $(2)$. 

Hence, $\{u_{n}\}$ is bounded in $\X_{\e}$ and we may assume that $u_{n}\rightharpoonup u$ in $\X_{\e}$ and $u_{n}\rightarrow u$ a.e. in $\R^{N}$.

\noindent
$(ii)$ Since $\langle \I'_{\e}(u_{n}), u^{-}_{n}\rangle=o_{n}(1)$, where $u_{n}^{-}=\min\{u_{n}, 0\}$, and $f(t)=0$ for $t\leq 0$, we have that 
\begin{align*}
&\int_{\R^{N}} |\nabla u_{n}|^{p-2}\nabla u_{n}\cdot \nabla u^{-}_{n}\,dx + \int_{\R^{N}} |\nabla u_{n}|^{q-2}\nabla u_{n}\cdot \nabla u^{-}_{n}\,dx  \\
&+ \int_{\R^{N}} V(\e x) (|u_{n}|^{p-2} u_{n}\,  +|u_{n}|^{q-2} u_{n})\, u_{n}^{-} \,dx=o_{n}(1),
\end{align*}
from which it follows
$$
\|u_{n}^{-}\|_{V, p}^{p}+\|u_{n}^{-}\|_{V, q}^{q}= o_{n}(1),
$$
that is $u_{n}^{-}\ri 0$ in $\X_{\e}$.
Moreover, $\{u_{n}^{+}\}$ is bounded in $\X_{\e}$. Now, we prove that $\I_{\e}(u^{+}_{n})\ri c$ and $\I'_{\e}(u^{+}_{n})=o_{n}(1)$. Clearly, $\|u_{n}\|_{V, t}=\|u_{n}^{+}\|_{V, t}+o_{n}(1)$ for $t\in \{p, q\}$. On the other hand, by \eqref{growthF}, the mean value theorem, and since $u_{n}= u_{n}^{+}+ u_{n}^{-}$, we have
\begin{align*}
\left|\int_{\R^{N}} F(u_{n})\, dx-\int_{\R^{N}} F(u^{+}_{n})\, dx\right|&\leq C\int_{\R^{N}} (|u_{n}|^{p-1}+|u_{n}|^{r-1})|u_{n}^{-}|\, dx \\
&\leq C|u_{n}^{-}|_{p}+C|u_{n}^{-}|_{r}\leq C\|u_{n}^{-}\|_{V,p}+C\|u_{n}^{-}\|_{V,q}\leq C\|u_{n}^{-}\|_{\e}=o_{n}(1).
\end{align*}
This shows that $\I_{\e}(u^{+}_{n})\ri c$. Next, we claim that $\I'_{\e}(u^{+}_{n})=o_{n}(1)$. Fix $\varphi\in \X_{\e}$ such that $\|\varphi\|_{\e}\leq 1$. Then we have 
\begin{align*}
&\left|\langle \I'_{\e}(u_{n}), \varphi\rangle-\langle \I'_{\e}(u^{+}_{n}), \varphi\rangle\right|\\
&=\Bigl| \int_{\R^{N}} [|\nabla u_{n}|^{p-2}\nabla u_{n}- |\nabla u_{n}^{+}|^{p-2} \nabla u_{n}^{+}] \nabla \varphi \,dx +\int_{\R^{N}} [|\nabla u_{n}|^{q-2}\nabla u_{n}- |\nabla u_{n}^{+}|^{q-2} \nabla u_{n}^{+}] \nabla \varphi \,dx \\
&+\int_{\R^{N}} V(\e x) [(|u_{n}|^{p-2} u_{n}\,  +|u_{n}|^{q-2} u_{n})- (|u_{n}^{+}|^{p-2} u_{n}^{+}\,  +|u_{n}^{+}|^{q-2} u_{n}^{+})]\, \varphi \,dx\\
&-\int_{\R^{N}} [f(u_{n}) - f(u_{n}^{+})]\varphi \, dx \Bigr|. 
\end{align*}
Now, recalling that for all $\xi>0$ there exists $C_{\xi}>0$ such that
\begin{align*}
||a+b|^{t-2}(a+b) - |a|^{t-2}a|\leq \xi |a|^{t-1} + C_{\xi} |b|^{t-1} \quad \mbox{ for all } a, b\in \R^{N} \mbox{ and } t>1, 
\end{align*}
we see that for $t\in \{p, q\}$ the following holds
\begin{align*}
&\Bigl| \int_{\R^{N}} [|\nabla u_{n}|^{t-2}\nabla u_{n}- |\nabla u_{n}^{+}|^{t-2} \nabla u_{n}^{+}] \nabla \varphi \,dx\Bigr|\\
&\leq \xi |\nabla u_{n}^{+}|_{t}^{t-1} |\nabla \varphi|_{t} + C_{\xi} |\nabla u_{n}^{-}|_{t}^{t-1} |\nabla \varphi|_{t}\\
&\leq \xi C + C'_{\xi} \|u_{n}^{-}\|_{\e}^{t-1}. 
\end{align*}
Consequently, 
\begin{align*}
\limsup_{n\ri \infty} \Bigl| \int_{\R^{N}} [|\nabla u_{n}|^{t-2}\nabla u_{n}- |\nabla u_{n}^{+}|^{t-2} \nabla u_{n}^{+}] \nabla \varphi \,dx\Bigr|\leq \xi C 
\end{align*}
and by the arbitrariness of $\xi>0$ we get 
\begin{align*}
\lim_{n\ri \infty} \int_{\R^{N}} [|\nabla u_{n}|^{t-2}\nabla u_{n}- |\nabla u_{n}^{+}|^{t-2} \nabla u_{n}^{+}] \nabla \varphi \,dx=0. 
\end{align*}
A similar argument shows that 
\begin{align*}
\lim_{n\ri \infty}\int_{\R^{N}} V(\e x) [(|u_{n}|^{p-2} u_{n}\,  +|u_{n}|^{q-2} u_{n})- (|u_{n}^{+}|^{p-2} u_{n}^{+}\,  +|u_{n}^{+}|^{q-2} u_{n}^{+})]\, \varphi \,dx=0. 
\end{align*}
Observing that 
\begin{align*}
\left| \int_{\R^{N}} [f(u_{n}) - f(u_{n}^{+})]\varphi \, dx \right|&= \left| \int_{\R^{N}} f(u^{-}_{n})\varphi \, dx \right| \\
&\leq C \int_{\R^{N}} (|u_{n}^{-}|^{p-1} + |u_{n}^{-}|^{r-1} ) |\varphi|\, dx\\
&\leq C (|u_{n}^{-}|_{p}^{p-1} |\varphi|_{p} + |u_{n}^{-}|_{r}^{r-1} |\varphi|_{r})  \\
&\leq C (\|u_{n}^{-}\|_{\e}^{p-1} + \|u_{n}^{-}\|_{\e}^{r-1} ) =o_{n}(1), 
\end{align*}
we can deduce that $\left|\langle \I'_{\e}(u_{n}), \varphi\rangle-\langle \I'_{\e}(u^{+}_{n}), \varphi\rangle\right|=o_{n}(1)$. Since $\langle\I'_{\e}(u_{n}), \varphi\rangle= o_{n}(1)$, we conclude that $\I'_{\e}(u_{n}^{+})= o_{n}(1)$. 
\end{proof}

\noindent
Since $f$ is only continuous, the next results are very important because they allow us to overcome the non-differentiability of $\N_{\e}$. We begin by proving some properties of the functional $\I_{\e}$.
\begin{lemma}\label{SW1}
Under assumptions $(V)$ and $(f_1)$-$(f_5)$, for any $\e>0$ we have:
\begin{compactenum}[$(i)$]
\item $\I'_{\e}$ maps bounded sets of $\X_{\e}$ into bounded sets of $\X_{\e}$.
\item $\I'_{\e}$ is weakly sequentially continuous in $\X_{\e}$.
\item $\I_{\e}(t_{n}u_{n})\rightarrow -\infty$ as $t_{n}\rightarrow \infty$, where $u_{n}\in K$ and $K\subset \X_{\e}\setminus\{0\}$ is a compact subset.
\end{compactenum}
\end{lemma}

\begin{proof}
$(i)$ Let $\{u_{n}\}$ be a bounded sequence in $\X_{\e}$ and $v \in \X_{\e}$. Then from assumptions $(f_{2})$ and $(f_{3})$ we can deduce that
\begin{align*}
\langle \I'_{\e}(u_{n}), v \rangle &\leq C_{1} \|u_{n}\|_{\e}^{p-1} \|v\|_{\e} + C_{2} \|u_{n}\|_{\e}^{q-1} \|v\|_{\e}+ C_{3} \|u_{n}\|_{\e}^{r-1} \|v\|_{\e} \leq C.
\end{align*}

\noindent
$(ii)$ 
Let $u_{n}\rightharpoonup u$ in $\X_{\e}$. By Lemma \ref{embedding}, we have that $u_{n}\rightarrow u$ in $L^{t}_{loc}(\R^{N})$ for all $t\in [1, q^{*}_{s})$ and $u_{n}\rightarrow u$ a.e. in $\R^{N}$. Then, for all $v \in C^{\infty}_{c}(\R^{N})$, it follows from \eqref{growthf} and the dominated convergence theorem that
\begin{align}\label{PELLACCI1}
\langle \I'_{\e}(u_{n}), v\rangle\rightarrow \langle \I'_{\e}(u), v\rangle.
\end{align}
Since $C^{\infty}_{c}(\R^{N})$ is dense in $\X_{\e}$, we can take $\{v_{j}\}\subset C^{\infty}_{c}(\R^{N})$ such that $\|v_{j}-v\|_{\e}\rightarrow 0$ as $j\rightarrow \infty$. Note that \eqref{growthf} and Lemma \ref{embedding} yield
\begin{align*}
|\langle \I'_{\e}(u_{n}), v\rangle-\langle \I'_{\e}(u), v\rangle|&\leq |\langle \I'_{\e}(u_{n})-\I'_{\e}(u), v_{j}\rangle|+|\langle \I'_{\e}(u_{n})-\I'_{\e}(u), v-v_{j}\rangle| \\
&\leq |\langle \I'_{\e}(u_{n})-\I'_{\e}(u), v_{j}\rangle|+C\int_{\R^{N}} (|u_{n}|^{p-1}+|u|^{p-1}+|u_{n}|^{r-1}+|u|^{r-1}) |v-v_{j}|\, dx \\
&\leq|\langle \I'_{\e}(u_{n})-\I'_{\e}(u), v_{j}\rangle|+C\|v_{j}-v\|_{\e}.
\end{align*}
For any $\zeta>0$, fix $j_{0}\in \mathbb{N}$ such that $\|v_{j_{0}}-v\|_{\e}<\frac{\zeta}{2C}$. By \eqref{PELLACCI1} there is $n_{0}\in \mathbb{N}$ such that 
$$
 |\langle \I'_{\e}(u_{n})-\I'_{\e}(u), v_{j_{0}}\rangle|<\frac{\zeta}{2} \quad \mbox{ for all } n\geq n_{0}.
$$
Thus 
$$
|\langle \I'_{\e}(u_{n}), v\rangle-\langle \I'_{\e}(u), v\rangle|<\zeta \quad \mbox{ for all } n\geq n_{0}
$$
and this shows that $\I'_{\e}$ is weakly sequentially continuous in $\X_{\e}$.

\noindent
$(iii)$ Without loss of generality, we may assume that $\|u\|_{\e}\leq 1$ for each $u\in K$. For $u_{n}\in K$, after passing to a subsequence, we obtain that $u_{n}\rightarrow u\in \mathbb{S}_{\e}$. Then, using $(f_{4})$ and Fatou's lemma, we can see that
\begin{align*}
\I_{\e}(t_{n}u_{n}) &=\frac{t_{n}^{p}}{p}\|u_{n}\|_{\e}^{p} + \frac{t_{n}^{q}}{q}\|u_{n}\|_{\e}^{q}- \int_{\R^{N}} F(t_{n}u_{n}) \, dx \\
&\leq t_{n}^{\vartheta} \left(\frac{\|u_{n}\|_{\e}^{p}}{t_{n}^{\vartheta-p}} + \frac{\|u_{n}\|_{\e}^{q}}{t_{n}^{\vartheta-q}} - \int_{\R^{N}} \frac{F(t_{n}u_{n})}{t_{n}^{\vartheta}} \, dx \right)\rightarrow -\infty \, \mbox{ as } n\rightarrow \infty.
\end{align*}
\end{proof}

\begin{lemma}\label{SW2}
Under the assumptions of Lemma \ref{SW1}, for $\e>0$ we have:
\begin{compactenum}[$(i)$]
\item for all $u\in \mathbb{S}_{\e}$, there exists a unique $t_{u}>0$ such that $t_{u}u\in \N_{\e}$. Moreover, $m_{\e}(u)=t_{u}u$ is the unique maximum of $\I_{\e}$ on $\X_{\e}$, where $\mathbb{S}_{\e}=\{u\in \X_{\e}: \|u\|_{\e}=1\}$.
\item The set $\N_{\e}$ is bounded away from $0$. Furthermore, $\N_{\e}$ is closed in $\X_{\e}$.
\item There exists $\alpha>0$ such that $t_{u}\geq \alpha$ for each $u\in \mathbb{S}_{\e}$ and, for each compact subset $W\subset \mathbb{S}_{\e}$, there exists $C_{W}>0$ such that $t_{u}\leq C_{W}$ for all $u\in W$.
\item For each $u\in \N_{\e}$, $m_{\e}^{-1}(u)=\frac{u}{\|u\|_{\e}}\in \N_{\e}$. In particular, $\N_{\e}$ is a regular manifold diffeomorphic to the sphere in $\X_{\e}$.
\item $c_{\e}=\inf_{\N_{\e}} \I_{\e}\geq \rho>0$ and $\I_{\e}$ is bounded below on $\N_{\e}$, where $\rho$ is independent of $\e$.
\end{compactenum}
\end{lemma}

\begin{proof}
$(i)$ The proof follows the same lines as the proof of Lemma \ref{newlemma}.

\noindent
$(ii)$ 
Using \eqref{growthf} and Lemma \ref{embedding}, for any $u\in \N_{\e}$ we have
\begin{align*}
\|u\|_{V, p}^{p}+ \|u\|_{V, q}^{q} = \int_{\R^{N}} f(u)u \, dx \leq \frac{\xi}{V_{0}} \|u\|_{V, p}^{p} + C_{\xi} \|u\|_{\e}^{r}. 
\end{align*}
Taking $\xi>0$ sufficiently small we can deduce that
\begin{align*}
C_{1}\|u\|_{V, p}^{p}+  \|u\|_{V, q}^{q} \leq C \|u\|_{\e}^{r}. 
\end{align*}
Now, if $\|u\|_{\e}\geq 1$, we are done. If $\|u\|_{\e}<1$, then $\|u\|_{V, p}^{p}\geq \|u\|_{V, p}^{q}$ so we get
\begin{align*}
C\|u\|_{\e}^{r} \geq C_{1}\|u\|_{V, p}^{p}+  \|u\|_{V, q}^{q} \geq C_{1}\|u\|_{V, p}^{q}+  \|u\|_{V, q}^{q} \geq C_{2} \|u\|_{\e}^{q}, 
\end{align*}
which implies that $\|u\|_{\e}\geq \kappa$ for some $\kappa>0$. \\
Next, we prove that $\N_{\e}$ is closed in $\X_{\e}$. Let $\{u_{n}\}\subset \N_{\e}$ be a sequence such that $u_{n}\rightarrow u$ in $\X_{\e}$. From Lemma \ref{SW1} we infer that $\I'_{\e}(u_{n})$ is bounded, so 
\begin{align*}
\langle \I'_{\e}(u_{n}), u_{n} \rangle - \langle \I'_{\e}(u), u \rangle= \langle \I'_{\e}(u_{n})- \I'_{\e}(u), u \rangle + \langle \I'_{\e}(u_{n}), u_{n}-u \rangle \rightarrow 0,
\end{align*}
that is $\langle \I'_{\e}(u), u\rangle=0$, which combined with $\|u\|_{\e}\geq \kappa$ implies that
$$
\|u\|_{\e}= \lim_{n\rightarrow \infty} \|u_{n}\|_{\e} \geq \kappa >0,
$$
hence $u\in \N_{\e}$.

\noindent
$(iii)$ For each $u\in \mathbb{S}_{\e}$ there exists $t_{u}>0$ such that $t_{u}u\in \N_{\e}$. Then, using $\|u\|_{\e}\geq \kappa$, we also have $t_{u}= \|t_{u}u\|_{\e}\geq \kappa$. 
It remains we prove that $t_{u}\leq C_{W}$ for all $u\in W\subset \mathbb{S}_{\e}$. We argue by contradiction: we suppose that
there exists a sequence $\{u_{n}\}\subset W\subset \mathbb{S}_{\e}$ such that $t_{u_{n}}\rightarrow \infty$. Since $W$ is compact, we can find $u\in W$ such that $u_{n}\rightarrow u$ in $\X_{\e}$ and $u_{n}\rightarrow u$ a.e. in $\R^{N}$. 

Now, using $(f_{4})$ we have
\begin{align*}
\I_{\e}(u)&= \I_{\e}(u)- \frac{1}{q}\langle \I'_{\e}(u), u\rangle \\
&=\left(\frac{1}{p}-\frac{1}{q}\right) |\nabla u|_{p}^{p} + \left(\frac{1}{p}-\frac{1}{q}\right)\int_{\R^{N}} V(\e x) |u|^{p} dx - \int_{\R^{N}} \left( F(u)-\frac{1}{q}f(u)u\right) \, dx\\
&=\left(\frac{1}{p}-\frac{1}{q}\right) \|u\|_{V, p}^{p} - \int_{\R^{N}} \left( F(u)-\frac{1}{q}f(u)u\right) \, dx\geq 0, 
\end{align*}
and this is in contrast with Lemma \ref{SW1}-$(iii)$ by which $\I_{\e}(t_{u_{n}}u_{n})\rightarrow -\infty$ as $n\rightarrow \infty$. 

\noindent
$(iv)$ Let us define the maps $\hat{m}_{\e}: \X_{\e}\setminus \{0\} \rightarrow \N_{\e}$ and $m_{\e}: \mathbb{S}_{\e}\rightarrow \N_{\e}$ by setting
\begin{align}\label{me}
\hat{m}_{\e}(u)= t_{u}u \quad \mbox{ and } \quad m_{\e}= \hat{m}_{\e}|_{\mathbb{S}_{\e}}.
\end{align}
In view of $(i)$-$(iii)$ and Proposition $3.1$ in \cite{SW} we can deduce that $m_{\e}$ is a homeomorphism between $\mathbb{S}_{\e}$ and $\N_{\e}$ and the inverse of $m_{\e}$ is given by $m_{\e}^{-1}(u)=\frac{u}{\|u\|_{\e}}$. Therefore $\N_{\e}$ is a regular manifold diffeomorphic to $\mathbb{S}_{\e}$.

\noindent
$(v)$ For $\e>0$, $t>0$ and $u\in \X_{\e}\setminus \{0\}$, we can see that \eqref{growthF} yields
\begin{align*}
\I_{\e}(tu)&\geq \frac{t^{p}}{p} |\nabla u|_{p}^{p} +\frac{t^{q}}{q} |\nabla u|_{q}^{q} + \int_{\R^{N}} V(\e x) \left( \frac{t^{p}}{p}|u|^{p}+ \frac{t^{q}}{q}|u|^{q}\right)\, dx - \frac{\xi t^{p}}{V_{0}} \int_{\R^{N}} V_{0}|u|^{p}\, dx - C_{\xi}t^{r} \int_{\R^{N}} |u|^{r}\, dx\\
&\geq \frac{t^{p}}{p} \left(1- \frac{\xi}{V_{0}}\right) \|u\|_{V, p}^{p}+ \frac{t^{q}}{q} \|u\|_{V, q}^{q} - C_{\xi} t^{r} \|u\|_{\e}^{r}
\end{align*}
so we can find $\rho>0$ such that $\I_{\e}(tu)\geq \rho>0$ for $t>0$ small enough. On the other hand, by using $(i)$-$(iii)$, we get (see \cite{SW}) that
\begin{align}\label{Nguyen1}
c_{\e}= \inf_{u\in \N_{\e}} \I_{\e}(u)= \inf_{u\in \X_{\e} \setminus \{0\}} \max_{t> 0} \, \I_{\e}(tu) = \inf_{u\in \mathbb{S}_{\e}} \max_{t> 0}\, \I_{\e}(tu)
\end{align}
which implies $c_{\e}\geq \rho$ and $\I_{\e}|_{\N_{\e}}\geq \rho$.
\end{proof}

\noindent
Now we introduce the following functionals $\hat{\Psi}_{\e}: \X_{\e} \setminus\{0\} \rightarrow \R$ and $\Psi_{\e}: \mathbb{S}_{\e}\rightarrow \R$ defined by
\begin{align*}
\hat{\Psi}_{\e}= \I_{\e}(\hat{m}_{\e}(u)) \quad \mbox{ and } \quad \Psi_{\e}= \hat{\Psi}_{\e}|_{\mathbb{S}_{\e}},
\end{align*}
where $\hat{m}_{\e}(u)= t_{u}u$ is given in \eqref{me}.
As in \cite{SW}, we have the following result:

\begin{lemma}\label{SW3}
Under the assumptions of Lemma \ref{SW1}, we have that for $\e>0$:
\begin{compactenum}[$(i)$]
\item $\Psi_{\e}\in C^{1}(\mathbb{S}_{\e}, \R)$, and
\begin{equation*}
\langle\Psi'_{\e}(w), v\rangle= \|m_{\e}(w)\|_{\e} \langle\I'_{\e}(m_{\e}(w)), v\rangle \quad \mbox{ for } v\in T_{w}(\mathbb{S}_{\e}).
\end{equation*}
\item $\{w_{n}\}$ is a Palais-Smale sequence for $\Psi_{\e}$ if and only if $\{m_{\e}(w_{n})\}$ is a Palais-Smale sequence for $\I_{\e}$. If $\{u_{n}\}\subset \N_{\e}$ is a bounded Palais-Smale sequence for $\I_{\e}$, then $\{m_{\e}^{-1}(u_{n})\}$ is a Palais-Smale sequence for $\Psi_{\e}$.
\item $u\in \mathbb{S}_{\e}$ is a critical point of $\Psi_{\e}$ if and only if $m_{\e}(u)$ is a critical point of $\I_{\e}$. Moreover, the corresponding critical values coincide and
\begin{equation*}
\inf_{\mathbb{S}_{\e}} \Psi_{\e}=\inf_{\N_{\e}} \I_{\e}=c_{\e}.
\end{equation*}
\end{compactenum}
\end{lemma}

\section{The autonomous problem}\label{Sect4}

\noindent
In this section we deal with the autonomous problem associated with \eqref{Pe}, that is
\begin{equation*}\tag{$AP_{\mu}$}\label{Pmu}
\left\{
\begin{array}{ll}
-\Delta_{p} u -\Delta_{q} u + \mu (|u|^{p-2}u + |u|^{q-2}u)  = f(u) &\mbox{ in } \R^{N} \\
u\in W^{1, p}(\R^{N})\cap W^{1, q}(\R^{N}), \quad 
u>0 \mbox{ in } \R^{N},  \mu>0.
\end{array}
\right.
\end{equation*}

The functional associated with \eqref{Pmu} is given by
\begin{equation*}
\J_{\mu}(u)= \frac{1}{p}|\nabla u|_{p}^{p}+\frac{1}{q}|\nabla u|_{q}^{q} + \mu \left[\frac{1}{p} |u|_{p}^{p} +\frac{1}{q}|u|_{q}^{q}\right]-\int_{\R^{N}} F(u) \, dx
\end{equation*}
which is well-defined on the space $\Y_{\mu}=W^{1, p}(\R^{N})\cap W^{1, q}(\R^{N})$ endowed with the norm
\begin{align*}
\|u\|_{\mu}= \|u\|_{\mu, p}+ \|u\|_{\mu, q},
\end{align*}
where
\begin{align*}
\|u\|_{\mu, t}^{t}= |\nabla u|_{t}^{t}+\mu |u|_{t}^{t}  \quad \mbox{ for all } t> 1.
\end{align*}
It is easy to check that $\J_{\mu}\in C^{1}(\Y_{\mu}, \R)$ and its differential is given by
\begin{align*}
\langle \J'_{\mu}(u), \varphi \rangle &=  \int_{\R^{N}} |\nabla u|^{p-2}\nabla u\cdot \nabla \varphi \,dx+ \int_{\R^{N}} |\nabla u|^{q-2}\nabla u\cdot \nabla \varphi \,dx \\
&\quad +\mu \left[\int_{\R^{N}} |u|^{p-2}u\,\varphi \, dx+ \int_{\R^{N}} |u|^{q-2}u\,\varphi \, dx \right]  - \int_{\R^{N}} f(u)\varphi \, dx
\end{align*}
for any $u, \varphi \in \Y_{\mu}$. Let us define the Nehari manifold associated with $\J_{\mu}$
\begin{equation*}
\M_{\mu} = \{u\in \Y_{\mu}\setminus \{0\} : \langle \J'_{\mu}(u), u\rangle =0\}.
\end{equation*}

We note that $(f_{4})$ yields
\begin{align}\label{coercive}
\J_{\mu}(u)&= \J_{\mu}(u)- \frac{1}{q}\langle \J'_{\mu}(u), u\rangle \nonumber \\
&=\left(\frac{1}{p}-\frac{1}{q}\right) \|u\|_{\mu, p}^{p} - \int_{\R^{N}} \left( F(u)-\frac{1}{q}f(u)u\right) \, dx \nonumber \\
&\geq \left(\frac{1}{p}-\frac{1}{q}\right) \|u\|_{\mu, p}^{p} \quad \mbox{ for all } u\in \M_{\mu}.
\end{align}
Arguing as in the previous section and using \eqref{coercive}, it is easy to prove the following lemma.

\begin{lemma}\label{SW2A}
Under the assumptions of Lemma \ref{SW1}, for $\mu>0$ we have:
\begin{compactenum}[$(i)$]
\item for all $u\in \mathbb{S}_{\mu}$, there exists a unique $t_{u}>0$ such that $t_{u}u\in \M_{\mu}$. Moreover, $m_{\mu}(u)=t_{u}u$ is the unique maximum of $\J_{\mu}$ on $\Y_{\mu}$, where $\mathbb{S}_{\mu}=\{u\in \Y_{\mu}: \|u\|_{\mu}=1\}$.
\item The set $\M_{\mu}$ is bounded away from $0$. Furthermore, $\M_{\mu}$ is closed in $\Y_{\mu}$.
\item There exists $\alpha>0$ such that $t_{u}\geq \alpha$ for each $u\in \mathbb{S}_{\mu}$ and, for each compact subset $W\subset \mathbb{S}_{\mu}$, there exists $C_{W}>0$ such that $t_{u}\leq C_{W}$ for all $u\in W$.
\item $\M_{\mu}$ is a regular manifold diffeomorphic to the sphere in $\Y_{\mu}$.
\item $d_{\mu}=\inf_{\M_{\mu}} \J_{\mu}>0$ and $\J_{\mu}$ is bounded below on $\M_{\mu}$ by some positive constant.
\item $\J_{\mu}$ is coercive on $\M_{\mu}$.
\end{compactenum}
\end{lemma}

\noindent
Now we define the following functionals $\hat{\Psi}_{\mu}: \Y_{\mu}\setminus\{0\} \rightarrow \R$ and $\Psi_{\mu}: \mathbb{S}_{\mu}\rightarrow \R$ by setting
\begin{align*}
\hat{\Psi}_{\mu}= \J_{\mu}(\hat{m}_{\mu}(u)) \quad \mbox{ and } \quad \Psi_{\mu}= \hat{\Psi}_{\mu}|_{\mathbb{S}_{\mu}}.
\end{align*}
Then we obtain the following result:

\begin{lemma}\label{SW3A}
Under the assumptions of Lemma \ref{SW1}, we have that for $\mu>0$:
\begin{compactenum}[$(i)$]
\item $\Psi_{\mu}\in C^{1}(\mathbb{S}_{\mu}, \R)$, and
\begin{equation*}
\langle \Psi'_{\mu}(w), v\rangle= \|m_{\mu}(w)\|_{\mu} \langle\J'_{\mu}(m_{\mu}(w)), v\rangle \quad \mbox{ for } v\in T_{w}(\mathbb{S}_{\mu}).
\end{equation*}
\item $\{w_{n}\}$ is a Palais-Smale sequence for $\Psi_{\mu}$ if and only if $\{m_{\mu}(w_{n})\}$ is a Palais-Smale sequence for $\J_{\mu}$. If $\{u_{n}\}\subset \M_{\mu}$ is a bounded Palais-Smale sequence for $\J_{\mu}$, then $\{m_{\mu}^{-1}(u_{n})\}$ is a Palais-Smale sequence for $\Psi_{\mu}$.
\item $u\in \mathbb{S}_{\mu}$ is a critical point of $\Psi_{\mu}$ if and only if $m_{\mu}(u)$ is a critical point of $\J_{\mu}$. Moreover, the corresponding critical values coincide and
\begin{equation*}
\inf_{\mathbb{S}_{\mu}} \Psi_{\mu}=\inf_{\M_{\mu}} \J_{\mu}=d_{\mu}.
\end{equation*}
\end{compactenum}
\end{lemma}

\begin{remark}
As in \eqref{Nguyen1}, invoking $(i)$-$(iii)$ of Lemma \ref{SW2A}, we can see that $d_{\mu}$ admits the following minimax characterization
\begin{align}\label{Nguyen2}
d_{\mu}= \inf_{u\in \M_{\mu}} \J_{\mu}(u)= \inf_{u\in \Y_{\mu}\setminus \{0\}} \max_{t> 0} \, \J_{\mu}(tu) = \inf_{u\in \mathbb{S}_{\mu}} \max_{t> 0} \, \J_{\mu}(tu).
\end{align}
\end{remark}

\begin{lemma}\label{lem2.2a}
Let $\{u_{n}\}\subset \M_{\mu}$ be a minimizing sequence for $\J_{\mu}$. Then $\{u_{n}\}$ is bounded in $\Y_{\mu}$ and there exist a sequence $\{y_{n}\}\subset \R^{N}$ and constants $R, \beta>0$ such that
\begin{equation*}
\liminf_{n\rightarrow \infty} \int_{\B_{R}(y_{n})} |u_{n}|^{q} dx \geq \beta >0.
\end{equation*}
\end{lemma}

\begin{proof}
Arguing as in the proof of Lemma \ref{lemB}, we can see that $\{u_{n}\}$ is bounded in $\Y_{\mu}$. Now, in order to prove the other assertion of this lemma, we argue by contradiction.
Assume that for any $R>0$ it holds
\begin{equation*}
\lim_{n\rightarrow \infty} \sup_{y\in \R^{N}} \int_{\B_{R}(y)} |u_{n}|^{q} dx=0.
\end{equation*}
Since $\{u_{n}\}$ is bounded in $\Y_{\mu}$, it follows by Lemma \ref{Lions} that
\begin{equation}\label{tv4N}
u_{n}\rightarrow 0 \mbox{ in } L^{t}(\R^{N}) \quad \mbox{ for any } t\in (q, \q).
\end{equation}
Fix $\xi \in (0, \mu)$. Then, taking into account that $\{u_{n}\}\subset \M_{\mu}$ and \eqref{growthf}, we have
\begin{align*}
0&= \langle \J'_{\mu}(u_{n}), u_{n} \rangle \\
&\geq |\nabla u_{n}|_{p}^{p}+ |\nabla u_{n}|_{q}^{q} + \mu \left[ |u_{n}|_{p}^{p} + |u_{n}|_{q}^{q}\right] - \xi |u_{n}|_{p}^{p} - C_{\xi} |u_{n}|_{r}^{r}\\
&\geq C_{1}\|u_{n}\|_{s, p}^{p}+C_{2} \|u_{n}\|_{s, q}^{q} -C_{3}|u_{n}|_{r}^{r},
\end{align*}
and in view of \eqref{tv4N}, we have that $\|u_{n}\|_{\mu}\ri 0$.
\end{proof}

\noindent
Next, we prove the following useful compactness result for the autonomous problem. For completeness, we recall that a critical point $u\neq 0$ of $\J_{\mu}$ satisfying $\J_{\mu}(u)=\inf_{\M_{\mu}}\J_{\mu}=d_{\mu}$ is called a ground state solution to \eqref{Pmu}; see chapter $4$ in \cite{W} for more details.
\begin{lemma}\label{lem4.3}
The problem \eqref{Pmu} has a positive ground state solution.
\end{lemma}

\begin{proof}
By virtue of $(v)$ of Lemma \ref{SW2A}, we know that $d_{\mu}>0$ for each $\mu>0$. Moreover, if $u\in \M_{\mu}$ satisfies $\J_{\mu}(u)=d_{\mu}$, then $m^{-1}_{\mu}(u)$ is a minimizer of $\Psi_{\mu}$ and it is a critical point of $\Psi_{\mu}$. In view of Lemma \ref{SW3A}, we can see that $u$ is a critical point of $\J_{\mu}$. Now we show that there exists a minimizer of $\J_{\mu}|_{\M_{\mu}}$. By Ekeland's variational principle \cite{W} there exists a sequence $\{\nu_{n}\}\subset \mathbb{S}_{\mu}$ such that $\Psi_{\mu}(\nu_{n})\rightarrow d_{\mu}$ and $\Psi'_{\mu}(\nu_{n})\rightarrow 0$ as $n\rightarrow \infty$. Let $u_{n}=m_{\mu}(\nu_{n}) \in \M_{\mu}$. Then, thanks to Lemma \ref{SW3A}, $\J_{\mu}(u_{n})\rightarrow d_{\mu}$ and $\J'_{\mu}(u_{n})\rightarrow 0$ as $n\rightarrow \infty$.
Therefore, arguing as in the proof of Lemma \ref{lemB}, $\{u_{n}\}$ is bounded in $\Y_{\mu}$ which is a reflexive space, so we may assume that $u_{n}\rightharpoonup u$ in $\Y_{\mu}$ for some $u\in \Y_{\mu}$.

It is clear that $\J_{\mu}'(u)=0$. Indeed, for all $\phi\in C^{\infty}_{c}(\R^{N})$,
\begin{align*}
&\int_{\R^{N}} |\nabla u_{n}|^{t-2}\nabla u_{n}\cdot \nabla \phi\, dx\rightarrow \int_{\R^{N}} |\nabla u|^{t-2}\nabla u\cdot \nabla\phi\, dx, \quad \mbox{ for } t\in \{p, q\}, \\
&\int_{\R^{N}} |u_{n}|^{t-2}u_{n}\phi\, dx\rightarrow \int_{\R^{N}} |u|^{t-2}u\phi\, dx, \quad \mbox{ for } t\in \{p, q\}, \\
&\int_{\R^{N}} f(u_{n})\phi \,dx\rightarrow \int_{\R^{N}} f(u)\phi \,dx,
\end{align*}
and using the fact that $\langle \J_{\mu}'(u_{n}),\phi \rangle=o_{n}(1)$, we can deduce that $\langle \J_{\mu}'(u),\phi \rangle=0$ for all $\phi\in C^{\infty}_{c}(\R^{N})$. By the density of $\phi\in C^{\infty}_{c}(\R^{N})$ in $\mathbb{Y}_{\mu}$, we obtain that $u$ is a critical point of $\J_{\mu}$.

Now, if $u\neq 0$, then $u$ is a nontrivial solution to \eqref{Pmu}.
Assume that $u=0$. Then $\|u_{n}\|_{\mu}\not \rightarrow 0$ in $\Y_{\mu}$. Hence, arguing as in the proof of Lemma \ref{lem2.2a} we can find a sequence $\{y_{n}\}\subset \R^{N}$ and constants $R, \beta>0$ such that
\begin{align}\label{tv1}
\liminf_{n\ri \infty} \int_{\B_{R}(y_{n})} |u_{n}|^{q} dx \geq \beta>0.
\end{align}
Now, let us define
\begin{align*}
\tilde{v}_{n}(x)=u_{n}(x+y_{n}).
\end{align*}
Due to the invariance by translations of $\R^{N}$, it is clear that $\|\tilde{v}_{n}\|_{\mu, t}=\|u_{n}\|_{\mu, t}$, with $t\in \{p, q\}$, so $\{\tilde{v}_{n}\}$ is bounded in $\Y_{\mu}$ and there exists $\tilde{v}$ such that $\tilde{v}_{n}\rightharpoonup \tilde{v}$ in $\Y_{\mu}$, $\tilde{v}_{n}\ri \tilde{v}$ in $L^{m}_{loc}(\R^{N})$ for any $m\in [1, \q)$ and $\tilde{v}\neq 0$ in view of \eqref{tv1}. Moreover, $\J_{\mu}(\tilde{v}_{n})= \J_{\mu}(u_{n})$ and $\J'_{\mu}(\tilde{v}_{n})=o_{n}(1)$, and arguing as before it is easy to check that $\J_{\mu}'(\tilde{v})=0$.

Now, say $u$ be the solution obtained before, and we prove that $u$ is a ground state solution. It is clear that $d_{\mu}\leq \J_{\mu}(u)$. On the other hand, by Fatou's lemma we can see that
\begin{align*}
\J_{\mu}(u)= \J_{\mu}(u)- \frac{1}{q}\langle \J_{\mu}'(u), u\rangle \leq \liminf_{n\ri \infty} \left[\J_{\mu}(u_{n})- \frac{1}{q}\langle \J_{\mu}'(u_{n}), u_{n}\rangle\right]= d_{\mu},
\end{align*}
which implies that $d_{\mu}=\J_{\mu}(u)$.

Finally, we prove that the ground state obtained earlier is positive.
Indeed, taking $u^{-}= \min \{u, 0\}$ as test function in \eqref{Pmu}, and applying $(f_{1})$ and invoking the following inequality
\begin{align*}
|x- y|^{t-2} (x- y) (x^{-} - y^{-}) \geq |x^{-}- y^{-}|^{t} \quad \forall t> 1,
\end{align*}
we can see that
\begin{align*}
\|u^{-}\|^{p}_{\mu, p} + \|u^{-}\|^{q}_{\mu, q} &\leq \int_{\R^{N}} |\nabla u|^{p-2} \nabla u\cdot \nabla u^{-} \, dxdy + \int_{\R^{N}} \mu |u|^{p-2} u u^{-} \, dx \\
&\quad+  \int_{\R^{N}} |\nabla u|^{q-2} \nabla u\cdot \nabla u^{-} \, dxdy  + \int_{\R^{N}} \mu |u|^{q-2} u u^{-} \, dx \\
&= \int_{\R^{N}} f(u)u^{-} \, dx =0,
\end{align*}
which implies that $u^{-}=0$, that is $u\geq 0$ in $\R^{N}$. By the regularity results in \cite{HL}, we have that $u\in L^{\infty}(\R^{N})\cap C^{1, \alpha}_{loc}(\R^{N})$ and $u(x)\ri 0$ as $|x|\ri \infty$ (in the exponential way). 
Applying the Harnack inequality in \cite{Trudinger}, we can see that $u>0$ in $\R^{N}$. 
This completes the proof of the lemma.
\end{proof}

\section{A first existence result for \eqref{P}}\label{Sect5}

\noindent
In this section we focus on the existence of a solution to \eqref{P} provided that $\e$ is sufficiently small. Let us start with the following useful lemma.
\begin{lemma}\label{lem2.2}
Let $\{u_{n}\}\subset \N_{\e}$ be a sequence such that $\I_{\e}(u_{n})\rightarrow c$ and $u_{n}\rightharpoonup 0$ in $\X_{\e}$. Then one of the following alternatives occurs:
\begin{compactenum}[$(a)$]
\item $u_{n}\rightarrow 0$ in $\X_{\e}$;
\item there are a sequence $\{y_{n}\}\subset \R^{N}$ and constants $R, \beta>0$ such that
\begin{equation*}
\liminf_{n\rightarrow \infty} \int_{\B_{R}(y_{n})} |u_{n}|^{q} dx \geq \beta >0.
\end{equation*}
\end{compactenum}
\end{lemma}

\begin{proof}
Assume that $(b)$ does not hold. Then, for any $R>0$, the following holds
\begin{equation*}
\lim_{n\rightarrow \infty} \sup_{y\in \R^{N}} \int_{\B_{R}(y)} |u_{n}|^{q} dx=0.
\end{equation*}
Since $\{u_{n}\}$ is bounded in $\X_{\e}$, it follows by Lemma \ref{Lions} that
\begin{equation}\label{tv4}
u_{n}\rightarrow 0 \mbox{ in } L^{t}(\R^{N}) \quad \mbox{ for any } t\in (q, \q).
\end{equation}
Now, we can argue as in the proof of Lemma \ref{lem2.2a} and deduce that $\|u_{n}\|_{\e}\rightarrow 0$ as $n\rightarrow \infty$.
\end{proof}

In order to get a compactness result for $\I_{\e}$, we need to prove the following auxiliary lemma.
\begin{lemma}\label{lem2.3}
Assume that $V_{\infty}<\infty$ and let $\{v_{n}\}\subset \N_{\e}$ be a sequence such that $\I_{\e}(v_{n})\rightarrow d$ with $v_{n}\rightharpoonup 0$ in $\X_{\e}$. If $v_{n}\not \rightarrow 0$ in $\X_{\e}$, then $d\geq d_{V_{\infty}}$, where $d_{V_{\infty}}$ is the infimum of $\J_{V_{\infty}}$ over $\M_{V_{\infty}}$.
\end{lemma}

\begin{proof}
Let $\{t_{n}\}\subset (0, \infty)$ be such that $\{t_{n}v_{n}\}\subset \M_{V_{\infty}}$. 
Our aim is to show that $\limsup_{n\ri \infty} t_{n}\leq 1$. \\
Assume by contradiction that there exist $\delta>0$ and a subsequence, denoted again by $\{t_{n}\}$, such that
\begin{equation}\label{tv8}
t_{n}\geq 1+\delta \quad \mbox{ for any } n\in \mathbb{N}.
\end{equation}
Since $\{v_{n}\}\subset \X_{\e}$ is a bounded $(PS)$ sequence for $\I_{\e}$, we have that $\langle \I'_{\e}(v_{n}), v_{n}\rangle =o_{n}(1)$, or equivalently
\begin{align}\label{tv3}
|\nabla v_{n}|_{p}^{p} + |\nabla v_{n}|_{q}^{q} + \int_{\R^{N}} V(\e x) |v_{n}|^{p} dx + \int_{\R^{N}} V(\e x) |v_{n}|^{q} dx - \int_{\R^{N}} f(v_{n})v_{n}\, dx = o_{n}(1).
\end{align}
Since $t_{n}v_{n}\in \M_{V_{\infty}}$, we also have that
\begin{align}\label{tv4}
&t_{n}^{p-q}|\nabla v_{n}|_{p}^{p} +  |\nabla v_{n}|_{q}^{q} + t_{n}^{p-q} V_{\infty} \int_{\R^{N}} |v_{n}|^{p} dx + V_{\infty} \int_{\R^{N}} |v_{n}|^{q} dx - \int_{\R^{N}} \frac{f(t_{n}v_{n})}{(t_{n}v_{n})^{q-1}} v_{n}^{q}\, dx =0.
\end{align}
Putting together \eqref{tv3} and \eqref{tv4}, we get
\begin{align}\label{tv5}
\int_{\R^{N}} \left(\frac{f(t_{n}v_{n})}{(t_{n}v_{n})^{q-1}}- \frac{f(v_{n})}{(v_{n})^{q-1}}\right) v_{n}^{q}\, dx \leq \int_{\R^{N}} (V_{\infty}-V(\e x)) |v_{n}|^{q} dx.
\end{align}
Now, using assumption \eqref{V0} we can see that, given $\zeta>0$, there exists $R=R(\zeta)>0$ such that
\begin{align}\label{hm}
V(\e x) \geq V_{\infty}-\zeta \quad \mbox{ for any } |x|\geq R.
\end{align}
From this, taking into account that $v_{n}\ri 0$ in $L^{q}(B_{R})$ and the boundedness of $\{v_{n}\}$ in $\X_{\e}$, we can infer
\begin{align}\label{tv6}
\int_{\R^{N}} &(V_{\infty}-V(\e x)) |v_{n}|^{q} dx = \int_{B_{R}(0)}(V_{\infty}-V(\e x)) |v_{n}|^{q} dx + \int_{\R^{N}\setminus B_{R}(0)}(V_{\infty}-V(\e x)) |v_{n}|^{q} dx  \nonumber \\
&\leq V_{\infty} \int_{B_{R}(0)} |v_{n}|^{q} dx + \zeta \int_{\R^{N}\setminus B_{R}(0)}|v_{n}|^{q} dx \nonumber \\
&\leq o_{n}(1) + \zeta C.
\end{align}
Combining \eqref{tv5} and \eqref{tv6}, we have
\begin{equation}\label{tv9}
\int_{\R^{N}} \left(\frac{f(t_{n}v_{n})}{(t_{n}v_{n})^{q-1}}- \frac{f(v_{n})}{(v_{n})^{q-1}}\right) v_{n}^{q}\, dx \leq o_{n}(1) + \zeta C.
\end{equation}
Since $v_{n}\not\rightarrow 0$ in $\X_{\e}$, we can apply Lemma \ref{lem2.2} to deduce the existence of a sequence $\{y_{n}\}\subset \R^{N}$ and two positive numbers $\bar{R}, \beta$ such that
\begin{align}\label{tv7}
\int_{\B_{\bar{R}}(y_{n})} |v_{n}|^{q}\, dx \geq \beta>0.
\end{align}
Let us consider $\tilde{v}_{n}=v_{n}(x+y_{n})$. Then we may assume that, up to a subsequence, $\tilde{v}_{n}\rightharpoonup \tilde{v}$ in $\X_{\e}$. By \eqref{tv7} there exists $\Omega\subset\R^{N}$ with positive measure and such that $\tilde{v}>0$ in $\Omega$. From \eqref{tv8}, $(f_{4})$ and \eqref{tv9}, we can infer that
\begin{equation*}
0<\int_{\Omega} \left(\frac{f((1+\delta)\tilde{v}_{n})}{((1+\delta)\tilde{v}_{n})^{q-1}}- \frac{f(\tilde{v}_{n})}{(\tilde{v}_{n})^{q-1}}\right) \tilde{v}_{n}^{q}\, dx \leq o_{n}(1) + \zeta C.
\end{equation*}
Taking the limit as $n\ri \infty$ and applying Fatou's lemma, we obtain
\begin{equation*}
0<\int_{\Omega} \left(\frac{f((1+\delta)\tilde{v})}{((1+\delta)\tilde{v})^{q-1}}- \frac{f(\tilde{v})}{(\tilde{v})^{q-1}}\right) \tilde{v}^{q}\, dx \leq \zeta C \quad \mbox{ for any } \zeta>0,
\end{equation*}
which is a contradiction.

Now we consider the following cases:

\textsc{Case 1:} Assume that $\limsup_{n\rightarrow \infty} t_{n}=1$. Thus there exists $\{t_{n}\}$ such that $t_{n}\rightarrow 1$. Taking into account that  $\I_{\e}(v_{n})\rightarrow c$, we have
\begin{align}\label{tv12new}
c+ o_{n}(1)&= \I_{\e}(v_{n})\nonumber \\
&=\I_{\e}(v_{n}) - \J_{V_{\infty}}(t_{n}v_{n})+ \J_{V_{\infty}}(t_{n}v_{n}) \nonumber \\
&\geq \I_{\e}(v_{n}) -\J_{V_{\infty}}(t_{n}v_{n}) + d_{V_{\infty}}.
\end{align}
Now, let us point out that
\begin{align}\begin{split}\label{tv12}
&\I_{\e}(v_{n}) -\J_{V_{\infty}}(t_{n}v_{n}) \\
&\quad = \frac{(1-t_{n}^{p})}{p}  |\nabla v_{n}|_{p}^{p} + \frac{(1-t_{n}^{q})}{q}  |\nabla v_{n}|_{q}^{q} + \frac{1}{p} \int_{\R^{N}} \left( V(\e x) - t_{n}^{p} V_{\infty}\right) |v_{n}|^{p} dx  \\
&\qquad +\frac{1}{q} \int_{\R^{N}} \left( V(\e x) - t_{n}^{q} V_{\infty}\right) |v_{n}|^{q} dx + \int_{\R^{N}} \left( F(t_{n} v_{n}) -F(v_{n}) \right) \, dx.
\end{split} \end{align}
Using condition \eqref{V0},  $v_{n}\rightarrow 0$ in $L^{p}(B_{R}(0))$, $t_{n}\rightarrow 1$, \eqref{hm}, and the fact that
\begin{align*}
V(\e x) - t_{n}^{p} V_{\infty} =\left(V(\e x) - V_{\infty} \right) + (1- t_{n}^{p}) V_{\infty}\geq -\zeta + (1- t_{n}^{p}) V_{\infty} \, \mbox{ for any } |x|\geq R,
\end{align*}
we get
\begin{align}\label{tv13}
&\int_{\R^{N}}  \left( V(\e x) - t_{n}^{p} V_{\infty}\right) |v_{n}|^{p} dx \nonumber \\
&= \int_{B_{R}(0)} \left( V(\e x) - t_{n}^{p} V_{\infty}\right) |v_{n}|^{p} dx+ \int_{\R^{N}\setminus B_{R}(0)} \left( V(\e x) - t_{n}^{p} V_{\infty}\right) |v_{n}|^{p} dx \nonumber \\
&\geq (V_{0}- t_{n}^{p}V_{\infty}) \int_{B_{R}(0)} |v_{n}|^{p} dx - \zeta \int_{\R^{N}\setminus B_{R}(0)} |v_{n}|^{p} dx+ V_{\infty}(1-t_{n}^{p}) \int_{\R^{N}\setminus B_{R}(0)} |v_{n}|^{p} dx \nonumber \\
&\geq o_{n}(1)- \zeta C.
\end{align}
In a similar fashion we can prove that
\begin{align}\label{tv131}
\int_{\R^{N}} & \left(V(\e x) - t_{n}^{q} V_{\infty}\right) |v_{n}|^{q} dx \geq o_{n}(1)- \zeta C.
\end{align}
Since $\{v_{n}\}$ is bounded in $\X_{\e}$, we can conclude that
\begin{align}\label{tv14}
\frac{(1-t_{n}^{p})}{p}  |\nabla v_{n}|_{p}^{p}= o_{n}(1) \quad \mbox{ and }\quad \frac{(1-t_{n}^{q})}{q}  |\nabla v_{n}|_{q}^{q}= o_{n}(1).
\end{align}
Thus, putting together \eqref{tv12}, \eqref{tv13}, \eqref{tv131} and \eqref{tv14}, we obtain
\begin{align}\label{tv15}
\I_{\e}(v_{n}) -\J_{V_{\infty}}(t_{n}v_{n}) \geq \int_{\R^{N}} \left( F(t_{n} v_{n}) -F(v_{n}) \right) \, dx +o_{n}(1)- \zeta C.
\end{align}
At this point, we aim to show that
\begin{align}\label{tv16}
\int_{\R^{N}} \left( F(t_{n} v_{n}) -F(v_{n}) \right) \, dx=o_{n}(1).
\end{align}
Applying the mean value theorem and \eqref{growthf}, we can deduce that
\begin{align*}
\int_{\R^{N}} | F(t_{n} v_{n}) -F(v_{n}) | \, dx \leq C|t_{n}-1|\int_{\R^{N}} |v_{n}|^{p} dx + C|t_{n}-1|\int_{\R^{N}} |v_{n}|^{r} dx.
\end{align*}
Exploiting the boundedness of $\{v_{n}\}$, we get the assertion. Gathering \eqref{tv12new}, \eqref{tv15} and \eqref{tv16}, we can infer that
\begin{align*}
c+ o_{n}(1)\geq o_{n}(1) - \zeta C + d_{V_{\infty}},
\end{align*}
and taking the limit as $\zeta \ri 0$ we get $c \geq d_{V_{\infty}}$.

\textsc{Case 2:} Assume that $\limsup_{n\rightarrow \infty} t_{n}=t_{0}<1$. Then there is a subsequence, still denoted by $\{t_{n}\}$, such that $t_{n}\rightarrow t_{0} (<1)$ and $t_{n}<1$ for any $n\in \mathbb{N}$.
Let us observe that
\begin{align}\label{tv17}
c+o_{n}(1)&= \I_{\e}(v_{n}) - \frac{1}{q}\langle \I'_{\e}(v_{n}), v_{n} \rangle \nonumber \\
&= \left(\frac{1}{p}- \frac{1}{q}\right) \|v_{n}\|_{V, p}^{p}  + \int_{\R^{N}} \left(\frac{1}{q}f(v_{n}) v_{n} - F(v_{n})\right) \,dx.
\end{align}
Recalling that $t_{n}v_{n}\in \M_{V_{\infty}}$, and using $(f_{5})$ and \eqref{tv17}, we obtain
\begin{align*}
d_{V_{\infty}} &\leq \J_{V_{\infty}}(t_{n}v_{n})  \\
&= \J_{V_{\infty}}(t_{n}v_{n}) - \frac{1}{q} \langle \J'_{V_{\infty}}(t_{n}v_{n}), t_{n}v_{n} \rangle \\
&=\left(\frac{1}{p}- \frac{1}{q}\right) \|t_{n}v_{n}\|_{V, p}^{p} +  \int_{\R^{N}} \left(\frac{1}{q} f(t_{n}v_{n}) t_{n}v_{n}- F(t_{n}v_{n})\right) \, dx \\
&\leq \left(\frac{1}{p}- \frac{1}{q}\right) \|v_{n}\|_{V, p}^{p} + \int_{\R^{N}} \left(\frac{1}{q}f(v_{n}) v_{n} - F(v_{n})\right) \,dx \\
&=c +o_{n}(1).
\end{align*}
Taking the limit as $n\rightarrow \infty$, we get $c\geq d_{V_{\infty}}$.
\end{proof}

\noindent
At this point we are able to prove the following compactness result.
\begin{proposition}\label{prop2.1}
Let $\{u_{n}\}\subset \N_{\e}$ be such that $\I_{\e}(u_{n})\rightarrow c$, where $c<d_{V_{\infty}}$ if $V_{\infty}<\infty$ and $c\in \R$ if $V_{\infty}=\infty$. Then $\{u_{n}\}$ has a convergent subsequence in $\X_{\e}$.
\end{proposition}

\begin{proof}
It is easy to see that $\{u_{n}\}$ is bounded in $\X_{\e}$. Then, up to a subsequence, we may assume that
\begin{align}\begin{split}\label{conv}
&u_{n}\rightharpoonup u \mbox{ in } \X_{\e}, \\
&u_{n}\rightarrow u \mbox{ in } L^{m}_{loc}(\R^{N}) \quad \mbox{ for any } m\in [1, \q), \\
&u_{n} \rightarrow u \mbox{ a.e. in } \R^{N}.
\end{split}\end{align}
By using assumptions $(f_{2})$-$(f_{3})$, \eqref{conv} and the fact that $\C^{\infty}_{c}(\R^{N})$ is dense in $\mathbb{X}_{\e}$, it is easy to check that $\I'_{\e}(u)=0$. \\
Now, let $v_{n}= u_{n}-u$.
By Lemma \ref{lem7}, we have
\begin{align}\label{tv19}
\I_{\e}(v_{n})
&=\I_{\e}(u_{n})-\I_{\e}(u)+o_{n}(1) \nonumber\\
&=c-\I_{\e}(u)+o_{n}(1)=d+o_{n}(1).
\end{align}
Now, we prove that $\I'_{\e}(v_{n})=o_{n}(1)$. For $t\in \{p, q\}$, by using Lemma \ref{lemVince} with $\eta_{n}= v_{n}$ and $w=u$, we get
\begin{equation}\label{D}
\iint_{\R^{2N}} |A(u_{n}) - A(v_{n}) - A(u)|^{t'} dx= o_{n}(1), 
\end{equation}
and arguing as in the proof of Lemma $3.3$ in \cite{MeW}, we can see that
\begin{equation}\label{C}
\int_{\R^{N}} V(\e x) ||v_{n}|^{t-2}v_{n}-|u_{n}|^{t-2}u_{n}+|u|^{t-2}u|^{t'} dx=o_{n}(1).
\end{equation}
Hence, by using the H\"older inequality, for any $\varphi\in \X_{\e}$ such that $\|\varphi\|_{\e}\leq 1$, we get
\begin{align*}
&|\langle \I'_{\e}(v_{n})-\I'_{\e}(u_{n})+\I'_{\e}(u), \varphi\rangle| \\
&\leq \left(\iint_{\R^{2N}}  |A(u_{n}) - A(v_{n}) - A(u)|^{p'} dx dy\right)^{\frac{1}{p'}} [\varphi]_{s, p} \\
&+ \left(\iint_{\R^{2N}}  |A(u_{n}) - A(v_{n}) - A(u)|^{q'} dx dy\right)^{\frac{1}{q'}} [\varphi]_{s, q} \\
&+\left(\int_{\R^{N}} V(\e x) ||v_{n}|^{p-2}v_{n}-|u_{n}|^{p-2}u_{n}+|u|^{p-2}u|^{p'} dx\right)^{p'} \left(\int_{\R^{N}} V(\e x) |\varphi|^{p}dx \right)^{\frac{1}{p}} \\
&+\left(\int_{\R^{N}} V(\e x) ||v_{n}|^{q-2}v_{n}-|u_{n}|^{q-2}u_{n}+|u|^{q-2}u|^{q'} dx\right)^{q'} \left(\int_{\R^{N}} V(\e x) |\varphi|^{q}dx \right)^{\frac{1}{q}} \\
&+\int_{\R^{N}} |(f(v_{n})-f(u_{n})+f(u)) \varphi| dx,
\end{align*}
and in view of $(iv)$ of Lemma \ref{lem7}, \eqref{D}, \eqref{C}, $\I'_{\e}(u_{n})=0$ and $\I'_{\e}(u)=0$ we obtain the assertion.

Now, we note that by using $(f_4)$ we can see that
\begin{equation}\label{tv199}
\I_{\e}(u)=\I_{\e}(u)-\frac{1}{q} \langle\I'_{\e}(u),u\rangle\geq 0.
\end{equation}
Assume $V_{\infty}<\infty$. It follows from \eqref{tv19} and \eqref{tv199} that
$$
d\leq c<d_{V_{\infty}}
$$
which together Lemma \ref{lem2.3} gives $v_{n}\rightarrow 0$ in $\X_{\e}$, that is $u_{n}\rightarrow u$ in $\X_{\e}$.\\
Let us consider the case $V_{\infty}=\infty$. Then, we can use Lemma \ref{lem6} to deduce that $v_{n}\rightarrow 0$ in $L^{m}(\R^{N})$ for all $m\in [p, \q)$. This, combined with assumptions $(f_2)$ and $(f_3)$, implies that
\begin{equation}\label{fn0}
\int_{\R^{N}} f(v_{n})v_{n}dx=o_{n}(1).
\end{equation}
Since $\langle\I'_{\e}(v_{n}), v_{n}\rangle=o_{n}(1)$, and applying \eqref{fn0} we can infer that
$$
\|v_{n}\|^{p}_{\e}=o_{n}(1),
$$
which yields $u_{n}\rightarrow u$ in $\X_{\e}$.
\end{proof}

\noindent
We conclude this section by giving the proof of the existence of a ground state solution to \eqref{Pe} (that is a nontrivial critical point $u$ of $\I_{\e}$ such that $\I_{\e}(u)=\inf_{\mathcal{N}_{\e}} \I_{\e}=c_{\e}$) whenever $\e>0$ is small enough.
\begin{theorem}\label{thm3.1}
Assume that $(V)$ and $(f_1)$-$(f_5)$ hold. Then there exists $\e_{0}>0$ such that, for any $\e\in (0, \e_{0})$, problem \eqref{Pe} admits a ground state solution.
\end{theorem}
\begin{proof}
By $(v)$ of Lemma \ref{SW2}, we know that $c_{\e}\geq \rho>0$ for each $\e>0$. Moreover, if $u_{\e}\in \N_{\e}$ satisfies $\I_{\e}(u_{\e})=c_{\e}$, then $m_{\e}^{-1}(u_{\e})$ is a minimizer of $\Psi_{\e}$ and it is a critical point of $\Psi_{\e}$. By virtue of Lemma \ref{SW3}, we can see that $u_{\e}$ is a critical point of $\I_{\e}$.
It remains to show that there exists a minimizer of $\I_{\e}|_{\N_{\e}}$. By Ekeland's variational principle \cite{W}, there exists a sequence $\{v_{n}\}\subset \mathbb{S}_{\e}$ such that $\Psi_{\e}(v_{n})\rightarrow c_{\e}$ and $\Psi'_{\e}(v_{n})\rightarrow 0$ as $n\rightarrow \infty$. Let $u_{n}=m_{\e}(v_{n}) \in \N_{\e}$. Then, by Lemma \ref{SW3}, we deduce that $\I_{\e}(u_{n})\rightarrow c_{\e}$, $\langle \I'_{\e}(u_{n}), u_{n}\rangle =0$ and $\I'_{\e}(u_{n})\rightarrow 0$ as $n\rightarrow \infty$.
Therefore, $\{u_{n}\}$ is a Palais-Smale sequence for $\I_{\e}$ at level $c_{\e}$.
It is easy to check that $\{u_{n}\}$ is bounded in $\X_{\e}$ and we denote by $u$ its weak limit. It is also easy to verify that $\I_{\e}'(u)=0$.

When $V_{\infty}=\infty$, by using Lemma \ref{lem6}, we have $\I_{\e}(u)=c_{\e}$ and $\I'_{\e}(u)=0$.\\
Now, we deal with the case $V_{\infty}<\infty$. In view of Proposition \ref{prop2.1} it is enough to show that $c_{\e}<d_{V_{\infty}}$ for small $\e$. Without loss of generality, we may suppose that
$$
V(0)=V_{0}=\inf_{x\in \R^{N}} V(x).
$$
Let $\mu\in \R$ be such that $\mu\in (V_{0}, V_{\infty})$. Clearly, $d_{V_{0}}<d_{\mu}<d_{V_{\infty}}$. 
Let us prove that there exists a function $w\in \Y_{\mu}$ with compact support such that
\begin{align}\label{D1}
\J_{\mu}(w) =\max_{t\geq 0} \J_{\mu}(tw) \quad \mbox{ and }\quad \J_{\mu}(w)<d_{V_{\infty}}.
\end{align}
Let $\psi\in C^{\infty}_{c}(\R^{N}, [0, 1])$ be such that $\psi=1$ in $B_{1}(0)$ and $\psi=2$ in $\R^{N}\setminus B_{2}(0)$. For any $R>0$, we set $\psi_{R}(x)= \psi(\frac{x}{R})$. We consider the function $w_{R}(x)= \psi_{R}(x)w^{\mu}(x)$, where $w^{\mu}$ is a ground state solution to \eqref{Pmu}. By the dominated convergence theorem we can see that
\begin{align}\label{FS1}
\lim_{R\ri \infty} \|w_{R}-w^{\mu}\|_{1,p} + \|w_{R}-w^{\mu}\|_{1,q} =0.
\end{align}
Let $t_{R}>0$ be such that $\J_{\mu}(t_{R}w_{R})= \max_{t\geq 0} \J_{\mu}(tw_{R})$. Then, $t_{R}w_{R}\in \M_{\mu}$.
Now there exists $\bar{r}>0$ such that $\J_{\mu}(t_{\bar{r}}w_{\bar{r}})<d_{V_{\infty}}$. Indeed, if $\J_{\mu}(t_{R}w_{R}) \geq d_{V_{\infty}}$ for any $R>0$, using $t_{R}w_{R}\in \M_{\mu}$, \eqref{FS1} and $w^{\mu}$ is a ground state, we can deduce that $t_{R}\ri 1$ and
\begin{align*}
d_{V_{\infty}}\leq \liminf_{R\ri \infty} \J_{\mu}(t_{R}w_{R}) = \J_{\mu}(w^{\mu})=d_{\mu}<d_{V_{\infty}},
\end{align*}
which gives a contradiction. Then, taking $w=\psi_{\bar{r}}w^{\mu}$, we can conclude that \eqref{D1} holds.

Now, by \eqref{V0}, we obtain that for some $\bar{\e}>0$
\begin{align}\label{D2}
V(\e x)\leq \mu \quad \mbox{ for all } x\in \supp w \mbox{ and } \e \in(0, \bar{\e}).
\end{align}
Then, in the light of \eqref{D1} and \eqref{D2}, we have for all  $\e\in (0,  \bar{\e})$
\begin{align*}
\max_{t> 0} \I_{\e}(tw)\leq \max_{t>0}\J_{\mu}(tw)= \J_{\mu}(w)<d_{V_{\infty}}.
\end{align*}
It follows from \eqref{Nguyen1} that $c_{\e}<d_{V_{\infty}}$ for all $\e\in (0,  \bar{\e})$. 
\end{proof}

\section{Multiple solutions for \eqref{P}}\label{Sect6}
This section is devoted to the study of the multiplicity of solutions to \eqref{P}. We begin by proving the following result which will be needed to implement the barycenter machinery.
\begin{proposition}\label{prop4.1}
Let $\e_{n}\rightarrow 0$ and $\{u_{n}\}\subset \N_{\e_{n}}$ be such that $\I_{\e_{n}}(u_{n})\rightarrow d_{V_{0}}$. Then there exists $\{\tilde{y}_{n}\}\subset \R^{N}$ such that the translated sequence
\begin{equation*}
v_{n}(x)=u_{n}(x+ \tilde{y}_{n})
\end{equation*}
has a subsequence which converges in $\Y_{V_{0}}$. Moreover, up to a subsequence, $\{y_{n}\}=\{\e_{n}\tilde{y}_{n}\}$ is such that $y_{n}\rightarrow y\in M$.
\end{proposition}

\begin{proof}
Since $\langle \I'_{\e_{n}}(u_{n}), u_{n} \rangle=0$ and $\I_{\e_{n}}(u_{n})\rightarrow d_{V_{0}}$, we know that $\{u_{n}\}$ is bounded in $\X_{\e}$. Since $d_{V_{0}}>0$, we can infer that $\|u_{n}\|_{\e_{n}}\not \rightarrow 0$.
Therefore, as in the proof of Lemma \ref{lem2.2}, we can find a sequence $\{\tilde{y}_{n}\}\subset \R^{N}$ and constants $R, \beta>0$ such that
\begin{equation}\label{tv21}
\liminf_{n\rightarrow \infty}\int_{B_{R}(\tilde{y}_{n})} |u_{n}|^{q} \,dx\geq \beta.
\end{equation}
Let us define
\begin{equation*}
v_{n}(x)=u_{n}(x+ \tilde{y}_{n}).
\end{equation*}
In view of the boundedness of $\{u_{n}\}$ and \eqref{tv21}, we may assume that $v_{n}\rightharpoonup v$ in $\Y_{V_{0}}$ for some $v\neq 0$.
Let $\{t_{n}\}\subset (0, \infty)$ be such that $w_{n}=t_{n} v_{n}\in \M_{V_{0}}$, and we set $y_{n}=\e_{n}\tilde{y}_{n}$.  \\
Thus, by using the change of variables $z\mapsto x+ \tilde{y}_{n}$, $V(x)\geq V_{0}$ and the invariance by translation, we can see that
\begin{align*}
d_{V_{0}}\leq \J_{V_{0}}(w_{n})\leq \I_{\e_{n}}(t_{n} v_{n})\leq \I_{\e_{n}}(u_{n})=d_{V_{0}}+ o_{n}(1).
\end{align*}
Hence we can infer $\J_{V_{0}}(w_{n})\rightarrow d_{V_{0}}$. This fact and $\{w_{n}\}\subset \M_{V_{0}}$  imply that there exists $K>0$ such that $\|w_{n}\|_{V_{0}}\leq K$ for all $n\in \mathbb{N}$.
Moreover, we can prove that the sequence $\{t_{n}\}$ is bounded in $\R$. In fact, $v_{n}\not \rightarrow 0$ in $\Y_{V_{0}}$, so there exists $\alpha>0$ such that $\|v_{n}\|_{V_{0}}\geq \alpha$. Consequently, for all $n\in \mathbb{N}$, we have
$$
|t_{n}|\alpha\leq \|t_{n}v_{n}\|_{V_{0}}=\|w_{n}\|_{V_{0}}\leq K,
$$
which yields $|t_{n}|\leq \frac{K}{\alpha}$ for all $n\in \mathbb{N}$.
Therefore, up to a subsequence, we may suppose that $t_{n}\rightarrow t_{0}\geq 0$. Let us show that $t_{0}>0$. Otherwise, if $t_{0}=0$, by the boundedness of $\{v_{n}\}$, we get $w_{n}= t_{n}v_{n} \rightarrow 0$ in $\Y_{V_{0}}$, that is $\J_{V_{0}}(w_{n})\rightarrow 0$ which is in contrast with the fact $d_{V_{0}}>0$. Thus $t_{0}>0$ and, up to a subsequence, we may assume that $w_{n}\rightharpoonup w= t_{0} v\neq 0$ in $\Y_{V_{0}}$. \\
Therefore
\begin{equation*}
\J_{V_{0}}(w_{n})\rightarrow d_{V_{0}} \quad \mbox{ and } \quad w_{n}\rightharpoonup w\neq 0 \mbox{ in } \Y_{V_{0}}.
\end{equation*}
From Lemma \ref{lem4.3}, we can deduce that $w_{n} \rightarrow w$ in $\Y_{V_{0}}$, that is $v_{n}\rightarrow v$ in $\Y_{V_{0}}$. \\
Now, we show that $\{y_{n}\}$ has a subsequence satisfying $y_{n}\rightarrow y\in M$. First, we prove that $\{y_{n}\}$ is bounded in $\R^{N}$.
Assume by contradiction that $\{y_{n}\}$ is not bounded, that is there exists a subsequence, still denoted by $\{y_{n}\}$, such that $|y_{n}|\rightarrow \infty$. \\
First, we deal with the case $V_{\infty}=\infty$. 
By using $\{u_{n}\}\subset \N_{\e_{n}}$ and by changing the variable, we can see that
\begin{align*}
&\int_{\R^{N}} V(\e_{n}x+y_{n})(|v_{n}|^{p} + |v_{n}|^{q}) dx \\
&\leq |\nabla v_{n}|^{p}_{p}+ |\nabla v_{n}|_{q}^{q} +\int_{\R^{N}} V(\e_{n}x+y_{n})(|v_{n}|^{p} + |v_{n}|^{q}) dx\\
&=\int_{\R^{N}} f(u_{n})u_{n} \, dx = \int_{\R^{N}} f(v_{n})v_{n} \, dx.
\end{align*}
By applying Fatou's lemma and $v_{n}\rightarrow v$ in $\Y_{V_{0}}$, we deduce that
\begin{align*}
\infty=\liminf_{n\rightarrow \infty} \int_{\R^{N}}  V(\e_{n}x+y_{n})(|v_{n}|^{p}+|v_{n}|^{q}) dx \leq \liminf_{n\rightarrow \infty} \int_{\R^{N}}  f(v_{n})v_{n} dx =\int_{\R^{N}}  f(v) v \,dx <\infty,
\end{align*}
which gives a contradiction. \\
Let us consider the case $V_{\infty}<\infty$. 
Taking into account that $w_{n}\rightarrow w$ strongly converges in $\Y_{V_{0}}$, condition \eqref{V0} and using the change of variable $z=x+ \tilde{y}_{n}$, we have
\begin{align}\label{tv22}
d_{V_{0}}&= \J_{V_{0}}(w) < \J_{V_{\infty}} (w) \nonumber\\
&\leq \liminf_{n\rightarrow \infty} \left [ \frac{1}{p} |\nabla w_{n}|^{p}_{p} + \frac{1}{q}|\nabla w_{n}|_{q}^{q}+ \int_{\R^{N}} V(\e_{n} x+y_{n})\left(\frac{1}{p}|w_{n}|^{p} + \frac{1}{q} |w_{n}|^{q}\right)\,dx-\int_{\R^{N}} F(w_{n}) \,dx \right]\nonumber \\
&=\liminf_{n\rightarrow \infty} \left[\frac{t^{p}_{n}}{p} |\nabla u_{n}|^{p}_{p} + \frac{t^{q}_{n}}{q} |\nabla u_{n}|^{q}_{q}+\int_{\R^{N}} V(\e_{n} z) \left(\frac{t^{p}_{n}}{p} |u_{n}|^{p}+ \frac{t^{q}_{n}}{q} |u_{n}|^{q}\right) dz-\int_{\R^{N}} F(t_{n} u_{n}) \, dz\right] \nonumber \\
&= \liminf_{n\rightarrow \infty} \I_{\e_{n}}(t_{n}u_{n}) \leq \liminf_{n\rightarrow \infty} \I_{\e_{n}} (u_{n})=d_{V_{0}}
\end{align}
which is a contradiction.
Thus $\{y_{n}\}$ is bounded and, up to a subsequence, we may assume that $y_{n}\rightarrow y$. If $y\notin M$, then $V_{0}<V(y)$ and we can argue as in \eqref{tv22} to get a contradiction. Therefore, we can conclude that $y\in M$.
\end{proof}

\noindent
Let $\delta>0$ be fixed and let $\psi\in C^{\infty}([0, \infty), [0, 1])$ be a nonincreasing function such that $\psi=1$ in $[0, \frac{\delta}{2}]$, $\psi=0$ in $[\delta, \infty)$ and $|\psi'|\leq C$ for some $C>0$. For any $y\in M$, we define
$$
\Upsilon_{\e, y}(x)=\psi(|\e x-y|) \omega\left(\frac{\e x-y}{\e}\right),
$$
where $\omega\in \mathbb{X}_{V_{0}}$ is a ground state solution to $(AP_{V_{0}})$ which exists by virtue of Lemma \ref{lem4.3}.

Let $t_{\e}>0$ be the unique positive number such that
$$
\I_{\e}(t_{\e}\Upsilon_{\e, y})=\max_{t\geq 0} \, \I_{\e}(t \Upsilon_{\e, y}).
$$
Define the map $\Phi_{\e}:M\rightarrow \N_{\e}$ by setting $\Phi_{\e}(y):=t_{\e} \Upsilon_{\e, y}$.
Then we can prove that
\begin{lemma}\label{lem4.1}
The functional $\Phi_{\e}$ satisfies the following limit
\begin{equation}\label{3.2}
\lim_{\e\rightarrow 0} \I_{\e}(\Phi_{\e}(y))=d_{V_{0}} \mbox{ uniformly in } y\in M.
\end{equation}
\end{lemma}
\begin{proof}
Assume by contradiction that there exist $\delta_{0}>0$, $\{y_{n}\}\subset M$ and $\e_{n}\rightarrow 0$ such that
\begin{equation}\label{4.41}
|\I_{\e_{n}}(\Phi_{\e_{n}} (y_{n}))-d_{V_{0}}|\geq \delta_{0}.
\end{equation}
Let us observe that the dominated convergence theorem implies
\begin{align}\label{4.421}
|\nabla \Upsilon_{\e_{n}, y_{n}}|_{p}^{p} + \int_{\R^{N}} V(\e_{n}x) |\Upsilon_{\e_{n}, y_{n}}|^{p} \, dx \ri |\nabla \omega|_{p}^{p} +\int_{\R^{N}} V_{0}|\omega|^{p}\, dx
\end{align}
and
\begin{align}\label{4.422}
|\nabla \Upsilon_{\e_{n}, y_{n}}|_{q}^{q} + \int_{\R^{N}} V(\e_{n}x) |\Upsilon_{\e_{n}, y_{n}}|^{q} \, dx \ri |\nabla \omega|_{q}^{q} +\int_{\R^{N}} V_{0}|\omega|^{q}\, dx.
\end{align}
Since $\langle \I'_{\e_{n}}(t_{\e_{n}}\Upsilon_{\e_{n}, y_{n}}), t_{\e_{n}} \Upsilon_{\e_{n}, y_{n}}\rangle=0$,
we can use the change of variable ${z=\frac{\e_{n}x-y_{n}}{\e_{n}}}$ to see that
\begin{align}\label{4.411}
&t_{\e_{n}}^{p}|\nabla\Upsilon_{\e_{n}, y_{n}}|_{p}^{p}+ t^{q}_{\e_{n}} |\nabla \Upsilon_{\e_{n}, y_{n}}|_{q}^{q} + \int_{\R^{N}} V(\e_{n}x) \left(  |t_{\e_{n}}\Upsilon_{\e_{n}, y_{n}}|^{p} +|t_{\e_{n}}\Upsilon_{\e_{n}, y_{n}}|^{q}\right) \, dx\nonumber \\
&=\int_{\R^{N}} f(t_{\e_{n}}\Upsilon_{\e_{n}}) t_{\e_{n}}\Upsilon_{\e_{n}} dx \nonumber\\
&=\int_{\R^{N}} f(t_{\e_{n}} \psi(|\e_{n}z|) \omega(z)) t_{\e_{n}} \psi(|\e_{n}z|) \omega(z) \, dz.
\end{align}
Now, we prove that $t_{\e_{n}}\rightarrow 1$. First we show that $t_{\e_{n}}\rightarrow t_{0}<\infty$. Assume by contradiction that $|t_{\e_{n}}|\rightarrow \infty$. Then, using the fact that $\psi(|x|)=1$ for $x\in B_{\frac{\delta}{2}}(0)$ and that $B_{\frac{\delta}{2}}(0)\subset B_{\frac{\delta}{2\e_{n}}}(0)$ for $n$ sufficiently large, we can see that \eqref{4.411} and $(f_5)$ imply
\begin{align}\label{4.44}
&t_{\e_{n}}^{p-q}|\nabla\Upsilon_{\e_{n}, y_{n}}|_{p}^{p}+ |\nabla \Upsilon_{\e_{n}, y_{n}}|_{q}^{q} + \int_{\R^{N}} V(\e_{n}x) \left(  t_{\e_{n}}^{p-q}|\Upsilon_{\e_{n}, y_{n}}|^{p} +|\Upsilon_{\e_{n}, y_{n}}|^{q}\right) \, dx\nonumber \\
&\quad \geq \int_{B_{\frac{\delta}{2}}(0)} \frac{f(t_{\e_{n}} \omega(z))}{(t_{\e_{n}} \omega(z))^{q-1}} (\omega(z))^{q} dz \geq \frac{f(t_{\e_{n}} \omega(\bar{z}))}{(t_{\e_{n}} \omega(\bar{z}))^{q-1}} \int_{B_{\frac{\delta}{2}}(0)} (\omega(z))^{q}dz
\end{align}
where $\bar{z}\in \R^{N}$ is such that $\omega(\bar{z})=\min\{\omega(z): |z|\leq \frac{\delta}{2}\}>0$ (note that $\omega\in C(\R^{N})$ and $\omega>0$ in $\R^{N}$).
Putting together $(f_4)$, $p<q$, $t_{\e_{n}}\rightarrow \infty$, \eqref{4.421} and \eqref{4.422}, we can see that \eqref{4.44} implies that $\|\Upsilon_{\e_{n}, y_{n}}\|_{V, q}^{q}\rightarrow \infty$, which gives a contradiction.
Therefore, up to a subsequence, we may assume that $t_{\e_{n}}\rightarrow t_{0}\geq 0$.
If $t_{0}=0$, we can use \eqref{4.421}, \eqref{4.422}, \eqref{4.411}, $p<q$ and $(f_2)$, to get
$$
\|\Upsilon_{\e_{n}, y_{n}}\|_{V, p}^{p}\rightarrow 0,
$$
which is a contradiction. Hence, $t_{0}>0$.
Now, we show that $t_{0}=1$.
Letting $n\rightarrow \infty$ in \eqref{4.411}, we can see that
\begin{align}\label{TVV1}
t_{0}^{p-q}|\nabla \omega|^{p}_{p}+ |\nabla \omega|^{q}_{q}+\int_{\R^{N}} V_{0} (t_{0}^{p-q}\omega^{p} dx + \omega^{q}) \, dx =\int_{\R^{N}} \frac{f(t_{0}\omega)}{(t_{0}\omega)^{q-1}} \omega^{q} \, dx.
\end{align}
Since $\omega\in  \M_{V_{0}}$ we have
\begin{align}\label{TVV2}
|\nabla \omega|^{p}_{p}+ |\nabla \omega|^{q}_{q}+\int_{\R^{N}} V_{0} (\omega^{p} dx + \omega^{q}) \, dx =\int_{\R^{N}} f(\omega)\omega\, dx.
\end{align}
Putting together \eqref{TVV1} and \eqref{TVV2}, we find
\begin{align}\label{TVV1}
(t_{0}^{p-q}-1)|\nabla \omega|^{p}_{p}+(t_{0}^{p-q}-1) \int_{\R^{N}} V_{0} \omega^{p}\, dx =\int_{\R^{N}} \left(\frac{f(t_{0}\omega)}{(t_{0}\omega)^{q-1}}- \frac{f(\omega)}{\omega^{q-1}}\right) \omega^{q}\, dx.
\end{align}
By $(f_{5})$, we can deduce that $t_{0}=1$.
This fact and the dominated convergence theorem yield
\begin{equation}\label{4.45}
\lim_{n\rightarrow \infty}\int_{\R^{N}} F(t_{\e_{n}} \Upsilon_{\e_{n}, y_{n}})\, dx=\int_{\R^{N}} F(\omega)\, dx.
\end{equation}
Hence, taking the limit as $n\rightarrow \infty$ in
\begin{align*}
\I_{\e_{n}}(\Phi_{\e_{n}}(y_{n}))=& \frac{t_{\e_{n}}^{p}}{p}|\nabla \Upsilon_{\e_{n}, y_{n}}|_{p}^{p}+ \frac{t_{\e_{n}}^{q}}{q}|\nabla \Upsilon_{\e_{n}, y_{n}}|_{q}^{q} \\
&+ \int_{\R^{N}} V(\e_{n}x) \left( \frac{t_{\e_{n}}^{p}}{p} |\Upsilon_{\e_{n}, y_{n}}|^{p} + \frac{t_{\e_{n}}^{q}}{q} |\Upsilon_{\e_{n}, y_{n}}|^{q}\right) \, dx\\
&  -\int_{\R^{N}} F(t_{\e_{n}} \Upsilon_{\e_{n}, y_{n}}) \, dx
\end{align*}
and exploiting \eqref{4.421}, \eqref{4.422} and \eqref{4.45}, we can deduce that
$$
\lim_{n\rightarrow \infty} \I_{\e_{n}}(\Phi_{\e_{n}}(y_{n}))=\J_{V_{0}}(\omega)=d_{V_{0}}
$$
which is impossible in view of \eqref{4.41}.
\end{proof}

\noindent
Now, we are in the position to introduce the barycenter map. We take $\rho>0$ such that $M_{\delta}\subset B_{\rho}(0)$, and we set $\chi: \R^{N}\rightarrow \R^{N}$ as follows
 \begin{equation*}
\chi(x)=
 \left\{
 \begin{array}{ll}
 x &\mbox{ if } |x|<\rho, \\
 \frac{\rho x}{|x|} &\mbox{ if } |x|\geq \rho.
  \end{array}
 \right.
 \end{equation*}
We define the barycenter map $\beta_{\e}: \N_{\e}\rightarrow \R^{N}$ by
\begin{align*}
\beta_{\e}(u)=\frac{ {\int_{\R^{N}} \chi(\e x) \left(|u|^{p}+ |u|^{q}\right) \, dx}}{ {\int_{\R^{N}} \left(|u|^{p}+ |u|^{q}\right) \, dx}}.
\end{align*}

\begin{lemma}\label{lem4.2}
The functional $\Phi_{\e}$ verifies the following limit
\begin{equation}\label{3.3}
\lim_{\e \rightarrow 0} \, \beta_{\e}(\Phi_{\e}(y))=y \mbox{ uniformly in } y\in M.
\end{equation}
\end{lemma}
\begin{proof}
Suppose by contradiction that there exist $\delta_{0}>0$, $\{y_{n}\}\subset M$ and $\e_{n}\rightarrow 0$ such that
\begin{equation}\label{4.4}
|\beta_{\e_{n}}(\Phi_{\e_{n}}(y_{n}))-y_{n}|\geq \delta_{0}.
\end{equation}
Using the definitions of $\Phi_{\e_{n}}(y_{n})$, $\beta_{\e_{n}}$, $\psi$ and the change of variable $z= \frac{\e_{n} x-y_{n}}{\e_{n}}$, we can see that
$$
\beta_{\e_{n}}(\Phi_{\e_{n}}(y_{n}))=y_{n}+\frac{\int_{\R^{N}}[\chi(\e_{n}z+y_{n})-y_{n}] (|\psi (|\e_{n}z|) \omega(z)|^{p}+ |\psi(|\e_{n}z|) \omega(z)|^{q})\, dz}{\int_{\R^{N}} (|\psi(|\e_{n}z|) \omega(z)|^{p}+ |\psi(|\e_{n}z|) \omega(z)|^{q})\, dz}.
$$
Taking into account $\{y_{n}\}\subset M\subset B_{\rho}(0)$ and applying the dominated convergence theorem, we can infer that
$$
|\beta_{\e_{n}}(\Phi_{\e_{n}}(y_{n}))-y_{n}|=o_{n}(1)
$$
which contradicts (\ref{4.4}).
\end{proof}

\noindent
At this point, we introduce a subset $\widetilde{\N}_{\e}$ of $\N_{\e}$ by taking a function $h:\R_{+}\rightarrow \R_{+}$ such that $h(\e)\rightarrow 0$ as $\e \rightarrow 0$, and setting
$$
\widetilde{\N}_{\e}=\{u\in \N_{\e}: \I_{\e}(u)\leq d_{V_{0}}+h(\e)\},
$$
where $h(\e)=\sup_{y\in M} |\I_{\e}(\Phi_{\e}(y))-d_{V_{0}}|$. By Lemma \ref{lem4.1}, we know that $h(\e)\ri 0$ as $\e \rightarrow 0$. By definition of $h(\e)$, we can deduce that for all $y\in M$ and $\e>0$, $\Phi_{\e}(y)\in \widetilde{\N}_{\e}$ and $\widetilde{\N}_{\e}\neq \emptyset$. Moreover, we have the following lemma.

\begin{lemma}\label{lem4.4}
For any $\delta>0$, the following holds
$$
\lim_{\e \rightarrow 0} \sup_{u\in \widetilde{\mathcal{N}}_{\e}} dist(\beta_{\e}(u), M_{\delta})=0.
$$
\end{lemma}

\begin{proof}
Let $\e_{n}\rightarrow 0$ as $n\rightarrow \infty$. For any $n\in \mathbb{N}$, there exists $\{u_{n}\}\subset \widetilde{\N}_{\e_{n}}$ such that
$$
\sup_{u\in \widetilde{\N}_{\e_{n}}} \inf_{y\in M_{\delta}}|\beta_{\e_{n}}(u)-y|=\inf_{y\in M_{\delta}}|\beta_{\e_{n}}(u_{n})-y|+o_{n}(1).
$$
Therefore, it suffices to prove that there exists $\{y_{n}\}\subset M_{\delta}$ such that
\begin{equation}\label{3.13}
\lim_{n\rightarrow \infty} |\beta_{\e_{n}}(u_{n})-y_{n}|=0.
\end{equation}
Thus, recalling that $\{u_{n}\}\subset  \widetilde{\N}_{\e_{n}}\subset  \N_{\e_{n}}$, we can deduce that
$$
d_{V_{0}}\leq c_{\e_{n}}\leq  \I_{\e_{n}}(u_{n})\leq d_{V_{0}}+h(\e_{n})
$$
which implies that $\I_{\e_{n}}(u_{n})\rightarrow d_{V_{0}}$. By Proposition \ref{prop4.1}, there exists $\{\tilde{y}_{n}\}\subset \R^{N}$ such that $y_{n}=\e_{n}\tilde{y}_{n}\in M_{\delta}$ for $n$ sufficiently large. Thus
$$
\beta_{\e_{n}}(u_{n})=y_{n}+\frac{ {\int_{\R^{N}}[\chi(\e_{n}z+y_{n})-y_{n}] (|u_{n}(z+\tilde{y}_{n})|^{p}+ |u_{n}(z+\tilde{y}_{n})|^{q}) \, dz}}{ {\int_{\R^{N}} (|u_{n}(z+\tilde{y}_{n})|^{p}+|u_{n}(z+\tilde{y}_{n})|^{q}) \, dz}}.
$$
Since $u_{n}(\cdot+\tilde{y}_{n})$ strongly converges in $\Y_{V_{0}}$ and $\e_{n}z+y_{n}\rightarrow y\in M$, we can deduce that $\beta_{\e_{n}}(u_{n})=y_{n}+o_{n}(1)$, that is (\ref{3.13}) holds.
\end{proof}

\noindent
Now we show that \eqref{Pe} admits at least $cat_{M_{\delta}}(M)$ solutions.
In order to achieve our aim, we recall the following result for critical points involving Lyusternik-Shnirel'man category. For more details one can see \cite{CL}.
\begin{theorem}\label{LSt}
Let $U$ be a $C^{1,1}$ complete Riemannian manifold (modelled on a Hilbert space). Assume that $h\in C^{1}(U, \R)$  is bounded from below and satisfies $-\infty<\inf_{U} h<d<k<\infty$. Moreover, suppose that $h$ satisfies the Palais-Smale condition on the sublevel $\{u\in U: h(u)\leq k\}$ and that $d$ is not a critical level for $h$. Then
$$
card\{u\in h^{d}: \nabla h(u)=0\}\geq cat_{h^{d}}(h^{d}), 
$$
where $h^{d}=\{u\in U \, : \, h(u)\leq d\}$. 
\end{theorem}

\noindent
With a view to apply Theorem \ref{LSt}, the following abstract lemma provides a very useful tool since relates the topology of some sublevel of a functional to the topology of some subset of the space $\R^{N}$; see \cite{CL}. 

\begin{lemma}\label{lemma2.2CL}
Let $\Omega, \Omega_{1}$ and $\Omega_{2}$ be closed sets with $\Omega_{1}\subset \Omega_{2}$ and let $\pi: \Omega \ri \Omega_{2}$, $\psi: \Omega_{1}\ri \Omega$ be continuous maps such that $\pi \circ \psi$ is homotopically equivalent to the embedding $j: \Omega_{1}\ri \Omega_{2}$. Then $cat_{\Omega}(\Omega)\geq cat_{\Omega_{2}}(\Omega_{1})$. 
\end{lemma}

\noindent
Since $\mathcal{N}_{\e}$ is not a $C^{1}$ submanifold of $\X_{\e}$, we cannot directly apply Theorem \ref{LSt}. Fortunately, by Lemma \ref{SW2}, we know that the mapping $m_{\e}$ is a homeomorphism between $\mathcal{N}_{\e}$ and $\mathbb{S}_{\e}$, and $\mathbb{S}_{\e}$ is a $C^{1}$ submanifold of $\X_{\e}$. So we can apply Theorem \ref{LSt} to
$\Psi_{\e}(u)=\I_{\e}(\hat{m}_{\e}(u))|_{\mathbb{S}_{\e}}=\I_{\e}(m_{\e}(u))$, where $\Psi_{\e}$ is given in Lemma \ref{SW3}.
In the light of the above observations, we are ready to give the proof of the main result of this work.
\begin{proof}[Proof of Theorem \ref{thmAI}]
For any $\e>0$, we define $\alpha_{\e} : M \rightarrow \mathbb{S}_{\e}$ by setting $\alpha_{\e}(y)= m_{\e}^{-1}(\Phi_{\e}(y))$. By using Lemma \ref{lem4.1} and the definition of $\Psi_{\e}$, we can see that
\begin{equation*}
\lim_{\e \rightarrow 0} \Psi_{\e}(\alpha_{\e}(y)) = \lim_{\e \rightarrow 0} \I_{\e}(\Phi_{\e}(y))= d_{V_{0}} \quad \mbox{ uniformly in } y\in M.
\end{equation*}
Set $\tilde{\mathbb{S}}_{\e}=\{ w\in \mathbb{S}_{\e} : \Psi_{\e}(w) \leq d_{V_{0}} + h(\e)\}$, where $h(\e)= \sup_{y\in M} |\Psi_{\e}(\alpha_{\e}(y)) - d_{V_{0}}|\ri 0$ as $\e \ri 0$. Thus, $\alpha_{\e}(y)\in \tilde{\mathbb{S}}_{\e}$ for all $y\in M$, and this yields $\tilde{\mathbb{S}}_{\e}\neq \emptyset$ for all $\e>0$. 

Taking into account Lemma \ref{lem4.1}, Lemma \ref{SW2}, Lemma \ref{SW3}, and Lemma \ref{lem4.4}, we can find $\bar{\e}= \bar{\e}_{\delta}>0$ such that the following diagram
\begin{equation*}
M\stackrel{\Phi_{\e}}{\rightarrow} \widetilde{\mathcal{N}}_{\e} \stackrel{m_{\e}^{-1}}{\rightarrow} \tilde{\mathbb{S}}_{\e} \stackrel{m_{\e}}{\rightarrow} \widetilde{\mathcal{N}}_{\e} \stackrel{\beta_{\e}}{\rightarrow} M_{\delta}
\end{equation*}
is well defined for any $\e \in (0, \bar{\e})$.
By using Lemma \ref{lem4.2}, there exists a function $\theta(\e, y)$ with $|\theta(\e, y)|<\frac{\delta}{2}$ uniformly in $y\in M$, for all $\e \in (0, \bar{\e})$, such that $\beta_{\e}(\Phi_{\e}(y))= y+ \theta(\e, y)$ for all $y\in M$. We can see that $H(t, y)= y+ (1-t)\theta(\e, y)$, with $(t, y)\in [0,1]\times M$, is a homotopy between $\beta_{\e} \circ \Phi_{\e}=(\beta_{\e}\circ m_{\e}) \circ \alpha_{\e}$ and the inclusion map $id: M \rightarrow M_{\delta}$. This fact and Lemma \ref{lemma2.2CL} imply that $cat_{\tilde{\mathbb{S}}_{\e}} (\tilde{\mathbb{S}}_{\e})\geq cat_{M_{\delta}}(M)$.
On the other hand, let us choose a function $h(\e)>0$ such that $h(\e)\rightarrow 0$ as $\e\rightarrow 0$ and such that $d_{V_{0}}+h(\e)$ is not a critical level for $\I_{\e}$. For $\e>0$ small enough, we deduce from Proposition \ref{prop2.1} that $\I_{\e}$ satisfies the Palais-Smale condition in $\widetilde{\N}_{\e}$. So, by $(ii)$ of Lemma \ref{SW3}, we infer that $\Psi_{\e}$ satisfies the Palais-Smale condition in $\tilde{\mathbb{S}}_{\e}$. Hence, by using Theorem \ref{LSt}, we obtain that $\Psi_{\e}$ has at least $cat_{\tilde{\mathbb{S}}_{\e}}(\tilde{\mathbb{S}}_{\e})$ critical points on $\tilde{\mathbb{S}}_{\e}$. Then, in view of $(iii)$ of Lemma \ref{SW3}, we can infer that $\I_{\e}$ admits at least $cat_{M_{\delta}}(M)$ critical points.
\end{proof}

\section{Concentration of solutions to \eqref{P}}\label{Sect7}
Let us start with the following result which plays a fundamental role in the study of the behavior of maximum points of solutions to \eqref{P}.
\begin{lemma}\label{lemMoser}
Let $v_{n}$ be a weak solution of the problem
\begin{equation}\label{Pvn}\tag{$P_{V_{n}}$}
\left\{
\begin{array}{ll}
-\Delta_{p} v_{n} -\Delta_{q} v_{n} + V_{n}(x)(|v_{n}|^{p-2}v_{n} + |v_{n}|^{q-2}v_{n}) = f(v_{n}) &\mbox{ in } \R^{N} \\
v_{n}\in W^{1, p}(\R^{N}) \cap W^{1, q}(\R^{N}), \, v_{n}>0 &\mbox{ in }  \R^{N},
\end{array}
\right.
\end{equation}
where $V_{n}(x)\geq V_{0}$ and $v_{n}\rightarrow v$ in $W^{1, p}(\R^{N})\cap W^{1, q}(\R^{N})$ for some $v\not\equiv 0$. Then $v_{n}\in L^{\infty}(\R^{N})$ and there exists $C>0$ such that $|v_{n}|_{\infty}\leq C$ for all $n\in \mathbb{N}$. Moreover, \begin{align*}
\lim_{|x|\rightarrow \infty} v_{n}(x)=0 \mbox{ uniformly in } n\in \mathbb{N}.
\end{align*}
\end{lemma}
\begin{proof}
We follow some ideas in \cite{AFans, HL} by developing a suitable Moser iteration argument \cite{Moser}.
For any $R>0$, $0<r\leq \frac{R}{2}$, let $\eta\in C^{\infty}(\R^{N})$ such that $0\leq \eta\leq 1$, $\eta=1$ in $\R^{N}\setminus B_{R}(0)$, $\eta=0$ in $\overline{B_{R-r}(0)}$  and $|\nabla \eta|\leq 2/r$. For each $n\in \mathbb{N}$ and for $L>0$,
let
\begin{equation*}
z_{L, n}=\eta^{q}v_{n} v_{L, n}^{q(\beta-1)} \quad \mbox{ and } \quad w_{L,n}=\eta v_{n} v_{L,n}^{\beta-1},
\end{equation*}
where  $v_{L,n}=\min\{v_{n}, L\}$ and $\beta>1$ to be determined later.
Choosing $z_{L, n}$ as a test function in \eqref{Pvn} we have 
\begin{align*}
\int_{\R^{N}}  |\nabla v_{n}|^{p-2} \nabla v_{n}\cdot \nabla z_{L,n} + |\nabla v_{n}|^{q-2} \nabla v_{n}\cdot \nabla z_{L,n} +V_{n} (v^{p-1}_{n}+v^{q-1}_{n}) z_{L,n} \, dx = \int_{\R^{N}} f(v_{n}) z_{L,n}\, dx.   
\end{align*}
By assumptions $(f_1)$ and $(f_2)$, for any $\xi>0$ there exists $C_{\xi}>0$ such that
\begin{equation*}
|f(t)|\leq \xi |t|^{p-1}+C_{\xi}|t|^{\q-1} \mbox{ for all } t\in \R.
\end{equation*}
Hence, using $(V_{1})$ and choosing $\xi\in (0, V_{0})$, we have
\begin{align*}
\int_{\R^{N}} \eta^{q} v_{L,n}^{q(\beta-1)} |\nabla v_{n}|^{q}\, dx \leq C_{\xi} \int_{\R^{N}} v_{n}^{\q} \eta^{q} v_{L,n}^{q(\beta-1)} \, dx -q\int_{\R^{N}} \eta^{q-1} v_{L,n}^{q(\beta-1)} v_{n} |\nabla v_{n}|^{q-2} \nabla v_{n}\cdot\nabla \eta \, dx. 
\end{align*}
For each $\tau>0$ we can use Young's inequality to obtain
\begin{align*}
\int_{\R^{N}} \eta^{q} v_{L,n}^{q(\beta-1)} |\nabla v_{n}|^{q}\, dx &\leq C_{\xi} \int_{\R^{N}} v_{n}^{\q} \eta^{q} v_{L,n}^{q(\beta-1)}\, dx + q\tau \int_{\R^{N}} |\nabla v_{n}|^{q} v_{L,n}^{q(\beta-1)} \eta^{q}\, dx \\
&\quad + q C_{\tau} \int_{\R^{N}} v_{n}^{q} |\nabla \eta|^{q} v_{L,n}^{q(\beta-1)}\, dx
\end{align*}
and taking $\tau>0$ sufficiently small, we get
\begin{align}\label{4.8HZ1}
\int_{\R^{N}} \eta^{q} v_{L,n}^{q(\beta-1)} |\nabla v_{n}|^{q}\, dx \leq C \int_{\R^{3}} v_{n}^{\q} \eta^{q} v_{L,n}^{q(\beta-1)}\, dx + C \int_{\R^{N}} |\nabla \eta|^{q} v_{n}^{q} v_{L,n}^{q(\beta-1)}\, dx. 
\end{align}
On the other hand, using the Sobolev inequality and the H\"older inequality, we can infer
\begin{align}\label{4.9HZ1}
|w_{L,n}|_{\q}^{q}&\leq C \int_{\R^{N}} |\nabla w_{L,n}|^{q}\, dx = C \int_{\R^{N}} |\nabla (\eta v_{L,n}^{\beta-1} v_{n})|^{q} \, dx\nonumber \\
&\leq C \beta^{q} \left( \int_{\R^{N}} |\nabla \eta|^{q} v_{n}^{q} v_{L,n}^{q(\beta-1)} \, dx + \int_{\R^{N}} \eta^{q} v_{L,n}^{q(\beta-1)} |\nabla v_{n}|^{q}\, dx \right). 
\end{align}
Combining \eqref{4.8HZ1} and \eqref{4.9HZ1}, we find
\begin{align}\label{5.7AFANS}
|w_{L,n}|_{\q}^{q}\leq C \beta^{q} \left( \int_{\R^{N}} |\nabla \eta|^{q} v_{n}^{q} v_{L,n}^{q(\beta-1)} \, dx + \int_{\R^{N}} v_{n}^{\q} \eta^{q} v_{L,n}^{q(\beta-1)} \, dx \right). 
\end{align}
We claim that $v_{n}\in L^{\frac{(\q)^{2}}{q}}(|x|\geq R)$ for $R$ large enough and uniformly in $n$. Let $\beta=\frac{\q}{q}$. From \eqref{5.7AFANS} we have 
\begin{align*}
|w_{L,n}|_{\q}^{q}\leq C \beta^{q} \left( \int_{\R^{N}} |\nabla \eta|^{q} v_{n}^{q} v_{L,n}^{\q-q} \, dx + \int_{\R^{N}} v_{n}^{\q} \eta^{q} v_{L,n}^{\q-q} \, dx \right)
\end{align*}
or equivalently
\begin{align*}
|w_{L,n}|_{\q}^{q}\leq C \beta^{q} \left( \int_{\R^{N}} |\nabla \eta|^{q} v_{n}^{q} v_{L,n}^{\q-q} \, dx + \int_{\R^{N}} v_{n}^{q} \eta^{q} v_{L,n}^{\q-q}v^{\q-q}_{n} \, dx \right).
\end{align*}
Using the H\"older inequality with exponents $\frac{\q}{q}$ and $\frac{\q}{\q-q}$, we obtain
\begin{align*}
|w_{L,n}|_{\q}^{q}\leq C \beta^{q} \left( \int_{\R^{N}} |\nabla \eta|^{q} v_{n}^{q} v_{L,n}^{\q-q} \, dx\right)+C\beta^{q} \left(\int_{\R^{N}} (v_{n}\eta v^{\frac{\q-q}{q}}_{L, n})^{\q} dx\right)^{\frac{q}{\q}}  \left(\int_{|x|\geq \frac{R}{2}} v^{\q}_{n}\, dx\right)^{\frac{\q-q}{\q}}.
\end{align*}
From the definition of $w_{L, n}$, we have
\begin{align*}
\left(\int_{\R^{N}} (v_{n}\eta v^{\frac{\q-q}{q}}_{L, n})^{\q} dx\right)^{\frac{q}{\q}} \leq C \beta^{q} \left( \int_{\R^{N}} |\nabla \eta|^{q} v_{n}^{q} v_{L,n}^{\q-q} \, dx\right)+C\beta^{q} \left(\int_{\R^{N}} (v_{n}\eta v^{\frac{\q-q}{q}}_{L, n})^{\q} dx\right)^{\frac{q}{\q}}  \left(\int_{|x|\geq \frac{R}{2}} v^{\q}_{n}\, dx\right)^{\frac{\q-q}{\q}}.
\end{align*}
Since $v_{n}\ri v$ in $W^{1, p}(\R^{N})\cap W^{1, q}(\R^{N})$, for $R>0$ sufficiently large, we get
\begin{align*}
\int_{|x|\geq \frac{R}{2}} v^{\q}_{n}\, dx\leq \epsilon \quad \mbox{ uniformly in } n\in \mathbb{N}.
\end{align*}
Hence,
\begin{align*}
\left(\int_{|x|\geq R} (v_{n}\eta v^{\frac{\q-q}{q}}_{L, n})^{\q} dx\right)^{\frac{q}{\q}} \leq C\beta^{q} \int_{\R^{N}} v^{q}_{n}v^{\q-q}_{L, n}\, dx\leq C\beta^{q} \int_{\R^{N}} v^{q}_{n}\, dx\leq K<\infty.
\end{align*}
Using Fatou's lemma, as $L\ri \infty$, we deduce that
$$
\int_{|x|\geq R} v^{\frac{(\q)^{2}}{q}}_{n}\, dx<\infty
$$
and therefore the assertion holds. Next, choosing $\beta=\q \frac{t-1}{qt}$ with $t=\frac{(\q)^{2}}{q(\q-q)}$, we have $\beta>1$, $\frac{qt}{t-1}<\q$ and $v_{n}\in L^{\frac{\beta q t}{t-1}}(|x|\geq R-r)$. From \eqref{5.7AFANS} we find
\begin{align*}
|w_{L,n}|_{\q}^{q}\leq C \beta^{q} \left( \int_{R\geq |x|\geq R-r} v_{n}^{q} v_{L,n}^{q(\beta-1)} \, dx + \int_{|x|\geq R-r} v_{n}^{\q}  v_{L,n}^{q(\beta-1)} \, dx \right) 
\end{align*}
or equivalently
\begin{align*}
|w_{L,n}|_{\q}^{q}\leq C \beta^{q} \left( \int_{R\geq |x|\geq R-r} v_{n}^{q\beta} \, dx + \int_{|x|\geq R-r} v_{n}^{\q-q}  v_{n}^{q\beta} \, dx \right). 
\end{align*}
Using the H\"older inequality with exponents $\frac{t}{t-1}$ and $t$,  we get
\begin{align*}
|w_{L,n}|_{\q}^{q}\leq C \beta^{q} \left\{\left[ \int_{R\geq |x|\geq R-r} v_{n}^{\frac{q\beta t}{t-1}} \, dx\right]^{\frac{t-1}{t}} \left[ \int_{R\geq |x|\geq R-r}dx \right]^{\frac{1}{t}} + \left[\int_{|x|\geq R-r} v_{n}^{(\q-q)t} \, dx \right]^{\frac{1}{t}}  \left[\int_{|x|\geq R-r} v_{n}^{\frac{q\beta t}{t-1}} \, dx \right]^{\frac{t-1}{t}} \right\}. 
\end{align*}
Since $(\q-q)t=(\q)^{2}$, we deduce that 
\begin{align*}
|w_{L,n}|_{\q}^{q}\leq C \beta^{q} \left( \int_{R\geq |x|\geq R-r} v_{n}^{\frac{q\beta t}{t-1}} \, dx\right)^{\frac{t-1}{t}}. 
\end{align*}
Note that
\begin{align*}
|v_{L,n}|_{L^{\q\beta}(|x|\geq R)}^{q\beta}&\leq \left(\int_{|x|\geq R-r} v^{\q \beta}_{L, n} \, dx\right)^{\frac{q}{\q}} \\
&\leq \left(\int_{\R^{N}} \eta^{q} v^{\q}_{n} v^{\q(\beta-1)}_{L, n} \, dx\right)^{\frac{q}{\q}}=|w_{L,n}|_{\q}^{q} \\
&\leq C \beta^{q} \left( \int_{R\geq |x|\geq R-r} v_{n}^{\frac{q\beta t}{t-1}} \, dx\right)^{\frac{t-1}{t}} \\
&=C\beta^{q} |v_{n}|^{\beta q}_{L^{\frac{q\beta t}{t-1}}(|x|\geq R-r)}
\end{align*}
which combined with Fatou's lemma with respect  to $L$ gives
\begin{align*}
|v_{n}|_{L^{\q\beta}(|x|\geq R)}^{q\beta}\leq C\beta^{q} |v_{n}|^{\beta q}_{L^{\frac{q\beta t}{t-1}}(|x|\geq R-r)}.
\end{align*}
Taking $\chi=\frac{\q(t-1)}{q t}$ and $s=\frac{q t}{t-1}$, it follows from the above inequality that 
\begin{align*}
|v_{n}|_{L^{\chi^{m+1}s}(|x|\geq R)}^{q\beta}\leq C^{\sum_{i=1}^{m} \chi^{-i}} \chi^{\sum_{i=1}^{m} i \chi^{-i}} |v_{n}|_{L^{\q}(|x|\geq R-r)}
\end{align*}
which implies that $|v_{n}|_{L^{\infty}(|x|\geq R)}\leq C |v_{n}|_{L^{\q}(|x|\geq R-r)}$. Since $v_{n}\ri v$ in $W^{1, q}(\R^{N})$, for all $\epsilon>0$ there exists $R>0$ such that 
$$
|v_{n}|_{L^{\infty}(|x|\geq R)}<\epsilon \quad \mbox{ for all } n\in \mathbb{N}.
$$
This completes the proof of the lemma.
\end{proof}

\begin{lemma}\label{UBlemAF}
There exists $\delta>0$ such that $|v_{n}|_{\infty}\geq \delta$ for all $n\in \mathbb{N}$.
\end{lemma}
\begin{proof}
Assume to the contrary that $|v_{n}|_{\infty}\rightarrow 0$ as $n\rightarrow \infty$. By $(f_2)$, there exists $n_{0}\in \mathbb{N}$ such that $\frac{f(|v_{n}|_{\infty}}{|v_{n}|_{\infty}^{p-1}}<\frac{V_{0}}{2}$ for all $n\geq n_{0}$.
Therefore, in view of $(f_{5})$, we can see that
\begin{align*}
|\nabla v_{n}|_{p}^{p} + |\nabla v_{n}|_{q}^{q}+V_{0}(|v_{n}|_{p}^{p}+ |v_{n}|_{q}^{q}) \leq \int_{\R^{N}} \frac{f(|v_{n}|_{\infty})}{|v_{n}|^{p-1}_{\infty}} |v_{n}|^{p} dx \leq \frac{V_{0}}{2} |v_{n}|^{p}_{p},
\end{align*}
which leads to a contradiction.
\end{proof}

\begin{proof}[End of the proof of Theorem \ref{thmAI}]
Let $u_{\e_{n}}$ be a solution to $(P_{\e_{n}})$.
Then $v_{n}(x)=u_{\e_{n}}(x+\tilde{y}_{n})$ is a solution to \eqref{Pvn} with $V_{n}(x)=V(\e_{n}x+\e_{n}\tilde{y}_{n})$, where $\{\tilde{y}_{n}\}$ is given by Proposition \ref{prop4.1}. Moreover,  in view of Proposition \ref{prop4.1}, up to subsequence, $v_{n}\rightarrow v\neq 0$ in $\Y_{V_{0}}$ and $y_{n}=\e_{n}\tilde{y}_{n}\rightarrow y\in M$. 
If $p_{n}$ denotes a global maximum point of $v_{n}$, we can use Lemma \ref{lemMoser} and Lemma \ref{UBlemAF} to see that $p_{n}\in B_{R}(0)$ for some $R>0$. Consequently, $z_{\e_{n}}=p_{n}+\tilde{y}_{n}$ is a global maximum point of $u_{\e_{n}}$, and then $\e_{n}z_{\e_{n}}=\e_{n}p_{n}+\e_{n}\tilde{y}_{n}\rightarrow y$ because $\{p_{n}\}$ is bounded. This fact and the continuity of $V$ yield $V(\e_{n}z_{\e_{n}})\rightarrow V(y)=V_{0}$ as $n\rightarrow \infty$.

Finally, we prove the exponential decay of $u_{\e_n}$. 
We use some arguments from \cite{HL}.
Since $v_{n}(x)\ri 0$ as $|x|\ri \infty$ uniformly in $n\in \mathbb{N}$, and using $(f_1)$, we can find $R>0$ such that
\begin{align*}
f(v_{n}(x))\leq \frac{V_{0}}{2}(v_{n}^{p-1}(x)+v_{n}^{q-1}(x)) \quad \mbox{ for all } |x|\geq R.
\end{align*}
Then, by using $(V_{1})$, we obtain
 \begin{align}\begin{split}\label{LIUGUOZAMP2+}
-\Delta_{p} v_{n}-\Delta_{q}v_{n}+\frac{V_{0}}{2}(v_{n}^{p-1}+v_{n}^{q-1}) &= f(v_{n})-\left(V_{n}-\frac{V_{0}}{2}\right)(v_{n}^{p-1}+v_{n}^{q-1}) \\
&\leq f(v_{n})-\frac{V_{0}}{2}(v_{n}^{p-1}+v_{n}^{q-1}) \leq 0 \quad \mbox{ for } |x|\geq R.
\end{split}\end{align}
Let $\phi(x)=M e^{-c|x|}$ with $c, M>0$ such that $c^{p}(p-1)< \frac{V_{0}}{2}$, $c^{q}(q-1)< \frac{V_{0}}{2}$ and $Me^{-cR}\geq v_{n}(x)$ for all $|x|=R$.
We can see that 
\begin{align}\label{LIUGUOZAMP1+}
&-\Delta_{p} \phi-\Delta_{q}\phi+\frac{V_{0}}{2}(\phi^{p-1}+\phi^{q-1}) \nonumber\\
&=\phi^{p-1}\left(\frac{V_{0}}{2}-c^{p}(p-1)+\frac{N-1}{|x|}c^{p-1} \right)+\phi^{q-1}\left(\frac{V_{0}}{2}-c^{q}(q-1)+\frac{N-1}{|x|}c^{q-1} \right)>0 \quad \mbox{ for } |x|\geq R.
\end{align}
Using $\eta=(v_{n}-\phi)^{+}\in W^{1, q}_{0}(\R^{N}\setminus B_{R})$ as a test function in \eqref{LIUGUOZAMP2+} and \eqref{LIUGUOZAMP1+}, we find 
\begin{align*}
0\geq & \int_{\{|x|\geq R\}\cap \{v_{n}>\phi \}} \left[ (|\nabla v_{n}|^{p-2}\nabla v_{n}- |\nabla \phi|^{p-2}\nabla \phi)\cdot\nabla \eta+(|\nabla v_{n}|^{q-2}\nabla v_{n} -|\nabla \phi|^{q-2}\nabla \phi) \cdot\nabla \eta\right] \\
&+\frac{V_{0}}{2} \left[(v_{n}^{p-1}-\phi^{p-1})+(v_{n}^{q-1}-\phi^{q-1})\right]\eta\, dx.
\end{align*}
Since for $t>1$ the following holds (see formula $(2.10)$ in \cite{Simon})
\begin{align*}
(|x|^{t-2}x-|y|^{t-2}y)\cdot (x-y)\geq 0  \quad\mbox{ for all } x, y\in \R^{N},
\end{align*}
and $U, v_{n}$ are continuous in $\R^{N}$, we deduce that $v_{n}(x)\leq \phi(x)$ for all $|x|\geq R$.
Recalling that $u_{\e_n}(x)=v_{n}(x-\tilde{y}_{n})$ and $\{p_{n}\}$ is bounded, we conclude that $u_{\e_n}(x)\leq C_{1}e^{-C_{2}|x-z_{\e_{n}}|}$ for all $x\in \R^{N}$. This completes the proof of Theorem \ref{thmAI}.
\end{proof}

\section*{Acknowledgments.}
The authors are very grateful to the anonymous referee for his/her careful reading of the manuscript and valuable suggestions that improved the presentation of the paper. 
The first author was partly supported by the GNAMPA Project 2020 entitled: {\it Studio Di Problemi Frazionari Nonlocali Tramite Tecniche Variazionali}.
The second author was partly supported by Slovenian research agency grants P1-0292, N1-0114, N1-0083, N1-0064 and J1-8131.

\renewcommand{\refname}{REFERENCES}

\end{document}